\documentclass[a4paper, 11pt]{article}
\usepackage[T1]{fontenc}
\usepackage{lmodern}
\usepackage[headings]{fullpage}               
\usepackage{amsfonts}
\usepackage{amsmath} 
\usepackage{amssymb}
\usepackage{graphicx}
\usepackage{amsthm}
\usepackage[dvipsnames,svgnames]{xcolor}
\usepackage[colorlinks=true,linkcolor=RoyalBlue,urlcolor=RoyalBlue,citecolor=PineGreen]{hyperref}
\usepackage[nameinlink]{cleveref}
\usepackage{tikz}
\usetikzlibrary{angles}
\usetikzlibrary{calc}
\usetikzlibrary {intersections,through}
\usepackage{enumerate}
\hypersetup{linktocpage=true}
\newtheorem{theorem}{Theorem}[section]
\newtheorem{lemma}[theorem]{Lemma}

\newtheorem{proposition}[theorem]{Proposition}

\theoremstyle{definition}\newtheorem{definition}[theorem]{Definition}
\theoremstyle{definition}\newtheorem{example}[theorem]{Example}
\theoremstyle{definition}
\theoremstyle{definition}

\newcommand{\rank}{\operatorname{rank}}

\newcommand{\Iso}{\operatorname{Isom}}

\newcommand{\Aut}{\operatorname{Aut}}

\definecolor{colR}{HTML}{CC6677}
\colorlet{colG}{DarkSeaGreen}

\tikzstyle{vertex}=[circle, draw=white, fill=black, inner sep=0pt, minimum size=4pt]
\tikzstyle{smallvertex}=[circle, line width=1.5pt, draw, fill=black, inner sep=0pt, minimum size=2pt]

\colorlet{ecol}{black!50!white}
\tikzstyle{edge}=[ecol,line width=1.5pt]
\tikzstyle{dedge}=[edge,-latex]
\tikzstyle{redge}=[edge,colR]
\tikzstyle{nedge}=[edge,colG]
\tikzstyle{oedge}=[edge,colR]
\tikzstyle{labelsty}=[font=\scriptsize]
\tikzstyle{axes}=[draw=black!50!white,-latex]

\usetikzlibrary{arrows}

\begin{document}

\title{Quotient graphs of symmetrically rigid frameworks}

\author{Sean Dewar\thanks{
Johann Radon Institute for Computational and Applied Mathematics (RICAM), Austrian Academy of Sciences, 4040 Linz, Austria. E-mail: {\tt sean.dewar@ricam.oeaw.ac.at}} \and
Georg Grasegger\thanks{
Johannes Kepler University Linz, Research Institute for Symbolic Computation, 4040 Linz, Austria. E-mail: {\tt georg.grasegger@jku.at}; } \and
Eleftherios Kastis\thanks{
Mathematics and Statistics, Lancaster University, Lancaster, LA1 4YF, UK. E-mail: {\tt l.kastis@lancaster.ac.uk}} \and
Anthony Nixon\thanks{
Mathematics and Statistics, Lancaster University, Lancaster, LA1 4YF, UK. E-mail: {\tt a.nixon@lancaster.ac.uk}}
}

\maketitle

\begin{abstract}
A natural problem in combinatorial rigidity theory concerns the determination of the rigidity or flexibility of bar-joint frameworks in $\mathbb{R}^d$ that admit some non-trivial symmetry. When $d=2$ there is a large literature on this topic. In particular, it is typical to quotient the symmetric graph by the group and analyse the rigidity of symmetric, but otherwise generic frameworks, using the combinatorial structure of the appropriate group-labelled quotient graph.
However, 
mirroring the situation for generic rigidity, little is known combinatorially when $d\geq 3$. Nevertheless in the periodic case, a key result of Borcea and Streinu \cite{borceastreinu11} characterises when a quotient graph can be lifted to a rigid periodic framework in $\mathbb{R}^d$. We develop an analogous theory for symmetric frameworks in $\mathbb{R}^d$. The results obtained apply to all finite and infinite 2-dimensional point groups, and then in arbitrary dimension they concern a wide range of infinite point groups, sufficiently large finite groups and groups containing translations and rotations. For the case of finite groups we also derive results concerning the probability of assigning group labels to a quotient graph so that the resulting lift is symmetrically rigid in $\mathbb{R}^d$.
\end{abstract} 

\section{Introduction}

A bar-joint \emph{framework} $(G,p)$ in $\mathbb{R}^d$ is an ordered pair consisting of a finite simple graph $G=(V,E)$ and a map $p:V\rightarrow \mathbb{R}^d$ (referred to as a \emph{placement} of $G$). The framework $(G,p)$ is \emph{rigid} if the only edge-length preserving continuous motions of the vertices arise from isometries of $\mathbb{R}^d$, and $(G,p)$ is \emph{flexible} otherwise.

Determining the rigidity of a given framework is a computationally challenging problem \cite{Abbot}.
Hence, most works in the combinatorial rigidity literature proceed by linearising and considering infinitesimal rigidity.
In particular the rigidity map $f_G:\mathbb{R}^{d|V|}\rightarrow \mathbb{R}^{|E|}$, for the framework $(G,p)$, is defined by putting $f_G(p)=(\|p(v)-p(w)\|^2 )_{vw\in E}$, where $\|\cdot \|$ denotes the Euclidean norm.
Then the Jacobean derivative matrix $df_G|_p$ (up to scale) is known as the rigidity matrix and $(G,p)$ is infinitesimally rigid if $G$ is complete on at most $d+1$ vertices or $df_G|_p$ has maximum rank ($=d|V|-\binom{d+1}{2}$).
In the \emph{generic} case, when the coordinates of $p$ form an algebraically independent set over $\mathbb{Q}$, infinitesimal rigidity is equivalent to rigidity \cite{AsimowRothI,AsimowRothII}, both depend only on the underlying graph, and one may apply matrix rank algorithms to test rigidity. 

When $d\leq 2$ the situation is even better and there are complete combinatorial descriptions of rigidity. More precisely, a folklore result says that a graph is rigid on the line if and only if it is connected, while Pollaczek-Geiringer \cite{PollaczekGeiringer} characterised generic rigidity in the plane. Her result is often referred to as Laman's theorem, since it was his paper \cite{Laman} that popularised the result. Pollaczek-Geiringer's characterisation leads quickly to fast deterministic algorithms for testing generic rigidity, see \cite{BergJordan03,PebbleGame,LS08}.
However, when $d\geq 3$ it remains an important open problem to understand rigidity in purely combinatorial terms. Details on this problem and further rigidity theoretic background may be found in \cite{GraverServatius,J16,Whiteley96}.

Many results in rigidity theory are motivated by applications, including those in sensor network localisation, protein structure determination and mechanical engineering. However, the genericity hypothesis is unrealistic in many applications, for example due to measurement error. In fact, in many applications the structures in question exhibit non-trivial symmetry.
This has motivated a number of groups to explore symmetric rigidity over the past two decades. We refer the reader to \cite{Bernstein20,CNSW20,ConnellyGuest,FowlerGuest00,IkeshitaTanigawa18,jordkasztani16,KKM21,MalesteinTheran15,OwenPower10,berndwalter10} and the references therein for details. 

In this article we consider \emph{forced symmetric rigidity}. That is, we take a graph $G$ with a non-trivial automorphism group $\Gamma$ and realise $G$ as a framework $(G,p)$ that is symmetric with respect to a particular geometric realisation of $\Gamma$. We then ask if $(G,p)$ has a symmetry preserving finite flex. As in the generic case, for symmetrically generic frameworks, this is equivalent to having a symmetry preserving infinitesimal flex, which can be expressed as a matrix condition using an analogue of the well-known rigidity matrix (details on this may be found in \cite{berndwalter10}). Complete combinatorial descriptions of symmetric, but otherwise generic, rigidity for various symmetry groups, including all cyclic groups, have been obtained in the plane \cite{jordkasztani16,KKM21,MalesteinTheran15}.
However, important open problems remain. For example, plane frameworks admitting even order dihedral symmetry groups are open. Moreover, when $d\geq 3$, characterising symmetric rigidity in combinatorial terms is a fundamental but wide open problem for all symmetry groups. 

In this article we prove combinatorial results both in $\mathbb{R}^2$ for arbitrary point groups (subgroups of the orthogonal group $O(d)$) and in general dimensions for infinite or sufficiently large point groups and for groups containing translations and rotations. To achieve this, the viewpoint we  adopt is inspired by a theorem of Borcea and Streinu \cite{borceastreinu11}. They studied periodic frameworks in $d$-dimensions and proved an analogue of Laman's theorem in all dimensions utilising additional freedom provided by the periodicity.
Their theorem, roughly speaking, says given a multigraph $G$, we can choose elements of $\mathbb{Z}^d$ to label the edges of $G$ such that the covering graph is periodically rigid in $\mathbb{R}^d$. An earlier result due to Whiteley \cite{ww88}, reformulated as \Cref{t:periodic} below, while not presented as a result about periodic frameworks, essentially considered the special case of periodic frameworks with fixed lattice representations and we give an alternative proof of that result using our new perspective.

In the context of symmetry groups we prove analogues of the theorem of Borcea and Streinu in $\mathbb{R}^2$ for arbitrary point groups. As mentioned, characterisations of forced-symmetric rigidity in the plane are already known for many finite point groups \cite{jordkasztani16,MalesteinTheran15} but our results include the remaining open finite groups as well as infinite point groups. We then prove analogues in $\mathbb{R}^d$ for a variety of `large' groups including translation groups, translation groups with additional point group symmetry, sufficiently dense point groups and sufficiently large point groups respectively. Our final main contribution concerns finite point groups. We present theoretical and computational results concerning the probability that a gain  assignment (labelling of the edges by elements of the symmetry group) results in a symmetrically rigid framework.

We conclude the introduction with a brief outline of what follows. In \Cref{sec:back} we introduce the relevant background on symmetric rigidity. Then, in
 \Cref{sec:plane-1iso,sec:plane-noiso}, we prove combinatorial results in 2-dimensions for all possible point groups using recursive construction techniques. We comment that, to the best of our knowledge, the case of infinite point groups has not previously been studied in rigidity theory. However the techniques in these sections do not seem to be amenable to higher dimensions. Nevertheless we extend the results in \Cref{sec:highd} to arbitrary dimensions for a wide variety of groups. The proofs here use classical combinatorial decomposition results and new geometric-analytic techniques. Here we briefly consider periodic frameworks and translational symmetry, and then develop analytic results for sufficiently dense symmetry groups and point groups of sufficient size. In the final section (\Cref{sec:prob}) we take a probabilistic approach to symmetric rigidity when the group is finite. After giving a number of examples and considering the effect of construction operations on the probability of a gain assignment giving a symmetrically rigid framework, we then prove that there exist infinitely many multigraphs whose probability of being assigned rigid gains, for any finite group, is positive but close to zero.

\section{Symmetric frameworks}
\label{sec:back}

We assume throughout that graphs have no loops or parallel edges,
while multigraphs allow both loops and parallel edges.
We mostly assume that $G$ has a finite set of vertices and edges.
Infinite graphs are allowed in certain circumstances, but these will be clearly designated.

Let $G=(V,E)$ be a multigraph.
We define $\vec{E}$ to be the set of all possible ordered triples $(e,v,w)$ where $e \in E$ and $v,w \in V$ are the source and sink of $e$. Let $\Gamma$ be a group.
A \emph{gain map} is a map $\phi:\vec{E} \rightarrow \Gamma$ where the following holds:
\begin{enumerate}[(i)]
	\item $\phi(e,v,w) = \phi(e,w,v)^{-1}$ for all $(e,v,w) \in \vec{E}$ where $v \neq w$,
	\item for every distinct pair of edges $e,e'$ between vertices $v,w$ we have $\phi(e,v,w) \neq \phi(e',v,w)$, and
	\item $\phi(e,v,v) \neq 1$ for every loop $e\in E$.
\end{enumerate}
We refer to the pair $(G,\phi)$ as a \emph{$\Gamma$-gain graph} or just simply a \emph{gain graph}.

The \emph{covering graph} of a $\Gamma$-symmetric gain graph $(G,\phi)$ is the graph $\mathbf{G} = (\mathbf{V},\mathbf{E})$ where $\mathbf{V} := V \times \Gamma$ and $\{(v,\gamma_v),(w,\gamma_w)\}$ is an edge if and only if there exists $(e,v,w)\in \vec{E}$ where $\phi(e,v,w) = \gamma_v^{-1} \gamma_w$.
We note that there is a unique map $\varphi : \Gamma \rightarrow \Aut (G)$ defined by its construction,
and $\mathbf{G}/\varphi = G$.
It is also immediate that $\mathbf{G}$ is finite if and only if $\Gamma$ is finite.

\begin{figure}[ht]
    \centering
    \begin{tikzpicture}
    \begin{scope}[every loop/.style={min distance=8mm,looseness=5}]
        \node[vertex,label={[labelsty]210:$a$}] (a) at (210:1.5) {};
        \node[vertex,label={[labelsty]-30:$b$}] (b) at (-30:1.5) {};
        \node[vertex,label={[labelsty]180:$c$}] (c) at (90:1.5) {};
        \draw[dedge] (b)to[bend left=15] node[below,labelsty] {0} (a);
        \draw[dedge] (a)to[bend left=15] node[above,labelsty] {$\pi/2$} (b);
        \draw[dedge] (b)to node[right,labelsty] {$\pi$} (c);
        \draw[dedge] (c)to[in=120,out=60,loop] node[above,labelsty] {$\pi/2$} (c);
        \draw[dedge] (c)to node[left,labelsty] {$-\pi/2$} (a);
    \end{scope}
    \begin{scope}[xshift=6cm]
        \draw[axes] (-2.5,0)--(2.5,0);
        \draw[axes] (0,-2.5)--(0,2.5);
        \foreach \r [count=\i] in {0,90,180,270}
        {
            \node[vertex,label={[labelsty]\r:$a_\i$}] (a\i) at (\r+70:1.2) {};
            \node[vertex,label={[labelsty]\r:$b_\i$}] (b\i) at (\r+20:0.4) {};
            \node[vertex,label={[labelsty]\r:$c_\i$}] (c\i) at (\r+85:2) {};
        }
        \foreach \i [remember=\i as \j (initially 4)] in {1,2,3,4}
        {
            \draw[edge] (a\i)edge(b\i);
            \draw[edge] (a\j)edge(b\i);
            \draw[edge] (c\i)edge(a\j);
            \draw[edge] (c\i)edge(c\j);
        }
        \draw[edge] (b1)edge(c3) (b2)edge(c4) (b3)edge(c1) (b4)edge(c2);
        
    \end{scope}
    \end{tikzpicture}
    \caption{A gain graph (left) and its covering graph (right) with respect to 4-fold rotational symmetry.}
    \label{fig:gaincovering}
\end{figure}
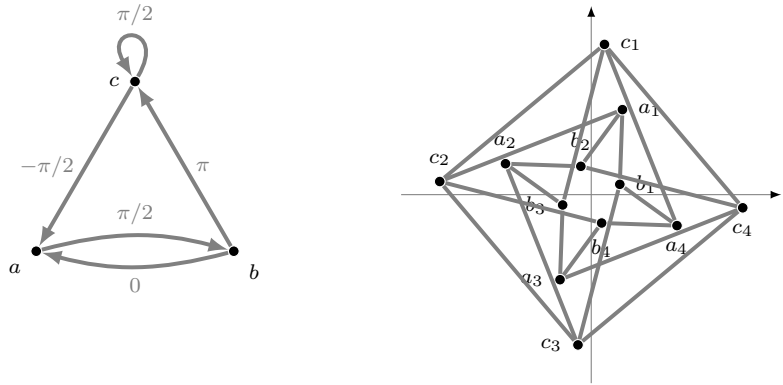

Given a $\Gamma$-symmetric gain graph $(G,\phi)$, we denote by $\mathbf{\tilde{V}}$ the set of vertex representatives $\tilde{v}=(v,1)$ for $v\in V$. Moreover, we fix an orientation on the edges of $G$, so that each edge $e\in E$ is an ordered pair $(v,w)$. For each edge $e=(v,w)\in E$, its edge representative $\tilde{e}$ be the unique edge in $\mathbf{E}$, given by $\tilde{e}=(v,1),(w,\gamma)$.  The set of edge representatives is denoted by $\mathbf{\tilde{E}}$.

\begin{definition}
Let $\Gamma$ be a subgroup of $\Iso(\mathbb{R}^d)$ and let $(G, \phi)$ be a $\Gamma$-gain graph with covering graph $\mathbf{G}$. We say that a placement $\mathbf{p} : \mathbf{V} \rightarrow \mathbb{R}^d$ of $\mathbf{G}$ is \emph{$\Gamma$-symmetric}  if  it satisfies
\[\mathbf{p}(v,\gamma) =\gamma \mathbf{p}(v,1),\,\text{ for all } \gamma\in \Gamma,\, v\in V.\] 
We also say that $(\mathbf{G},\mathbf{p})$ is a $\Gamma$-symmetric framework with gain graph $(G,\phi)$.
\end{definition}

Every isometry on $\mathbb{R}^d$ is the composition of a linear isometry followed by a translation.
For each $\gamma\in\Gamma$, we denote by $\gamma_\ell$ the linear isometry on $\mathbb{R}^d$ that is uniquely defined by the linear part of the affine isometry $\gamma$.

 Recall that a map $\mathbf{u} : \mathbf{V} \rightarrow \mathbb{R}^d$ is an \emph{infinitesimal flex} if it lies in the kernel of the associated \emph{rigidity matrix} $\mathbf{R}(\mathbf{G},\mathbf{p})$ (see, for instance, \cite{GraverServatius}).
 The flex $\mathbf{u}$ is \emph{trivial} if there exists a skew-symmetric matrix $T$ and vector $x$ so that $\mathbf{u}(v,\gamma) = T \mathbf{p}(v,\gamma) + x$ for all $v \in V$ and $\gamma \in \Gamma$.
 The flex $\mathbf{u}$ is \emph{$\Gamma$-symmetric} if it also satisfies $\mathbf{u}(v,\gamma) = \gamma_\ell \mathbf{u}(v,1)$ for all $v \in V$ and $\gamma \in \Gamma$.

\begin{definition}
Let	$(\mathbf{G},\mathbf{p})$ be a $\Gamma$-symmetric framework with gain graph $(G,\phi)$.
The framework $(\mathbf{G},\mathbf{p})$ is \emph{$\Gamma$-symmetrically rigid} if every $\Gamma$-symmetric infinitesimal flex $\mathbf{u} : \mathbf{V} \rightarrow \mathbb{R}^d$  is trivial. The covering graph $\mathbf{G}$ is \emph{$\Gamma$-symmetrically rigid} if there exists a $\Gamma$-symmetrically rigid placement of it, otherwise $\mathbf{G}$ is \emph{$\Gamma$-symmetrically flexible}.
\end{definition} 

The framework in \Cref{fig:gaincovering} is rigid, whereas \Cref{fig:gaincovering-flex} shows a flexible one.

\begin{figure}[ht]
    \centering
    \begin{tikzpicture}
    \begin{scope}[every loop/.style={min distance=8mm,looseness=5}]
        \node[vertex,label={[labelsty]210:$a$}] (a) at (210:1.5) {};
        \node[vertex,label={[labelsty]-30:$b$}] (b) at (-30:1.5) {};
        \node[vertex,label={[labelsty]180:$c$}] (c) at (90:1.5) {};
        \draw[dedge] (b)to[bend left=15] node[below,labelsty] {0} (a);
        \draw[dedge] (a)to[bend left=15] node[above,labelsty] {$\pi/2$} (b);
        \draw[dedge] (c)to[in=120,out=60,loop] node[above,labelsty] {$\pi/2$} (c);
        \draw[dedge] (c)to node[left,labelsty] {$-\pi/2$} (a);
    \end{scope}
    \begin{scope}[xshift=6cm]
        \draw[axes] (-2.5,0)--(2.5,0);
        \draw[axes] (0,-2.5)--(0,2.5);
        \foreach \w/\op in {0/1,-5/0.15}
        {
            \begin{scope}[opacity=\op]
            \foreach \r [count=\i] in {0,90,180,270}
            {
                \node[vertex] (c\i) at (\r+45:2) {};
            }
            \foreach \r [count=\i,remember=\i as \j (initially 4)] in {270,0,90,180}
            {
                \node[vertex,rotate around=\w:(c\i)] (a\j) at (\r+60:1.2) {};
            }
            \foreach \r [count=\i,remember=\i as \j (initially 4)] in {0,90,180,270}
            {
                \path[name path=circ1] (a\i) circle [radius=0.7cm];
                \path[name path=circ2] (a\j) circle [radius=1.1cm];
                \path[name intersections={of=circ1 and circ2}];
                \ifnum\i=2
                    \node[vertex] (b\i) at (intersection-2) {};
                \else
                    \node[vertex] (b\i) at (intersection-1) {};
                \fi
            }
            \foreach \i [remember=\i as \j (initially 4)] in {1,2,3,4}
            {
                \draw[edge] (a\i)edge(b\i);
                \draw[edge] (a\j)edge(b\i);
                \draw[edge] (c\i)edge(a\j);
                \draw[edge] (c\i)edge(c\j);
            }
            \end{scope}
        }
        \foreach \i in {1,2,3,4}
        {
            \node[label={[labelsty]90*\i-90:$a_\i$}] at (a\i) {};
            \node[label={[labelsty,label distance=-5pt]90*\i-45:$c_\i$}] at (c\i) {};
            \node[label={[labelsty,label distance=6pt]90*\i-90:$b_\i$}] at (b\i) {};
        }
    \end{scope}
    \end{tikzpicture}
    \caption{A gain graph (left) and its covering graph (right) with respect to 4-fold rotational symmetry yielding a flexible framework.}
    \label{fig:gaincovering-flex}
\end{figure}
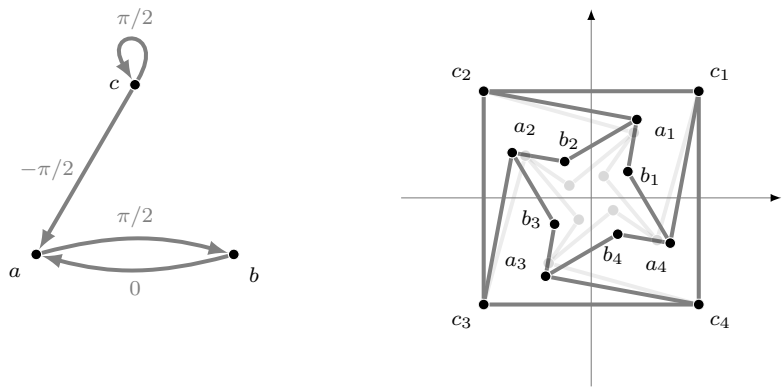
 
There is an alternative way to determine whether a $\Gamma$-symmetric framework $(\mathbf{G},\mathbf{p})$ is $\Gamma$-symmetrically rigid. Given a $\Gamma$-symmetric placement $\mathbf{p}$ of $\mathbf{G}$, define the \emph{orbit placement} $p: V\rightarrow \mathbb{R}^d,~ v \mapsto \mathbf{p}(v,1)$ of $(G,\phi)$,
and define the triple $(G,\phi,p)$ to be an \emph{orbit framework}. Define the \emph{orbit rigidity matrix} $R(G,\phi,p)$ to be the $|E| \times d|V|$ matrix with entries
\begin{align*}
	R(G,\phi,p)_{e,(v,j)}
	:= 
	\begin{cases}
		[p(v) - \gamma p(w)]_j &\text{if } (e,v,w) \in \vec{E}, ~ v \neq w \text{ and } \phi(e,v,w)=\gamma, \\
		[2p(v) - \gamma p(v)-\gamma^{-1}p(v)]_j &\text{if } (e,v,v) \in \vec{E} \text{ and } \phi(e,v,v)=\gamma, \\
		0 &\text{otherwise}.
	\end{cases}
\end{align*}

\begin{proposition}\label{p:gain} 
	Let $\Gamma \leq \Iso(\mathbb{R}^d)$  and $(G, \phi)$ be a $\Gamma$-gain graph with covering graph $\mathbf{G}$.
	Then $\mathbf{u}$ is a $\Gamma$-symmetric flex of $(\mathbf{G},\mathbf{p})$ if and only if $R(G,\phi,p)u =0$,
	where $u(v) := \mathbf{u}(v,1)$ for all $v \in V$.
\end{proposition}
\begin{proof}
Let $\mathbf{u}:\mathbf{V}\rightarrow \mathbb{R}^d$ be a vector such that $\mathbf{u}(v,\gamma) = \gamma_\ell \mathbf{u}(v,1)$ for all $v\in V$ and $\gamma\in \Gamma$. Fix some edge $\mathbf{e}=(v,\gamma),(w,\gamma')\in \mathbf{E}$ with $v\neq w$. Then 
\begin{align*}
(\mathbf{R}(\mathbf{G},\mathbf{p})\mathbf{u})_{\mathbf{e}}&=(\mathbf{p}(v,\gamma)-\mathbf{p}(w,\gamma'))\cdot
(\mathbf{u}(v,\gamma)-\mathbf{u}(w,\gamma'))\\
&=(\gamma_\ell (\mathbf{p}(v,1)-\mathbf{p}(w,\gamma^{-1}\gamma')))\cdot
(\gamma_\ell (\mathbf{u}(v,1)-\mathbf{u}(w,\gamma^{-1}\gamma')))\\
&=(\mathbf{p}(v,1)-\mathbf{p}(w,\gamma^{-1}\gamma'))\cdot
(\mathbf{u}(v,1)-\mathbf{u}(w,\gamma^{-1}\gamma'))\\
&=(\mathbf{p}(v,1)-\mathbf{p}(w,\gamma^{-1}\gamma'))\cdot \mathbf{u}(v,1)+
(\mathbf{p}(w,\gamma^{-1}\gamma')-\mathbf{p}(v,1))\cdot\mathbf{u}(w,\gamma^{-1}\gamma')\\
&=(\mathbf{p}(v,1)-\mathbf{p}(w,\gamma^{-1}\gamma'))\cdot \mathbf{u}(v,1)\\
&\qquad+
((\gamma^{-1}\gamma')_\ell(\mathbf{p}(w,1)-\mathbf{p}(v,\gamma'^{-1}\gamma)))\cdot((\gamma^{-1}\gamma')_\ell\mathbf{u}(w,1))
\\&=(\mathbf{p}(v,1)-\mathbf{p}(w,\gamma^{-1}\gamma'))\cdot \mathbf{u}(v,1)+
(\mathbf{p}(w,1)-\mathbf{p}(v,\gamma'^{-1}\gamma))\cdot\mathbf{u}(w,1)\\&=
(R(G,\phi,p)u)_{e},
\end{align*}
where $u(v) := \mathbf{u}(v,1)$ for all $v \in V$. We work similarly for $v=w$, so the result now follows.
\end{proof}    

We now define a map $u :V \rightarrow \mathbb{R}^d$ to be a \emph{$\Gamma$-symmetric infinitesimal flex} (or \emph{flex} for short) of $(G,\phi,p)$ if it lies in the kernel of the associated orbit rigidity matrix $R(G,\phi,p)$.
The flex $u$ is \emph{trivial} if the corresponding flex $\mathbf{u}$ of the framework $(\mathbf{G},\mathbf{p})$ is trivial.
With this we can now make the following definition of $\Gamma$-symmetric rigidity for the orbit framework $(G,\phi,p)$.

\begin{definition}
Let	$(G,\phi)$ be a $\Gamma$-gain graph and let $(G,\phi,p)$ be an orbit framework.
The orbit framework $(G,\phi,p)$ is \emph{$\Gamma$-symmetrically (infinitesimally) rigid} if every $\Gamma$-symmetric infinitesimal flex $u : V \rightarrow \mathbb{R}^d$ of $(G,\phi,p)$ is trivial. 
The gain graph $(G,\phi)$ is \emph{$\Gamma$-symmetrically rigid} if there exists a $\Gamma$-symmetrically rigid orbit placement of it, otherwise $(G,\phi)$ is \emph{$\Gamma$-symmetrically flexible}.
\end{definition} 

We define an orbit framework $(G,\phi,p)$ to be \emph{regular} if the rank of the orbit rigidity matrix is maximal over the set of orbit placements of $(G,\phi)$.
The set of regular orbit placements of a gain graph $(G,\phi)$ can be seen to be a Zariski open subset of the set of orbit placements.
It follows that either all orbit placements of a gain graph are $\Gamma$-symmetrically flexible,
or almost all orbit placements of a gain graph are $\Gamma$-symmetrically rigid.

Let $G=(V,E)$ be a (multi)graph and let $k,\ell,m$ be non-negative integers where $\ell \geq m$.
For $X\subset V$, let $i_G(X)$ denote the number of edges of $G$ in the subgraph induced by $X$.
We say $G$ is \emph{$(k,l)$-sparse} if $i(X)\leq k|X|-l$ for all $X\subset V$ with $|X|\geq k$ and $G$ is \emph{$(k,l)$-tight} if it is $(k,l)$-sparse and $|E|=k|V|-l$.
Now suppose $\phi$ is a $\Gamma$-symmetric gain map.
A subgraph $H \subset G$ is \emph{balanced} if for all closed walks 
\begin{align*}
	(v_1,e_{12},v_2,\ldots, v_{n-1} , e_{(n-1)n}, v_n)
\end{align*}
(i.e.~$e_{ij}$ has ends $v_i,v_j$ and $v_n=v_1$) in $H$ we have $\phi(e_{(n-1)n},v_{n-1},v_n) \ldots \phi(e_{12},v_{1},v_2) =1$.
We define $(G,\phi)$ to be \emph{$(k,\ell,m)$-gain-sparse} (respectively, \emph{$(k,\ell,m)$-gain-tight}) if $G$ is $(k,m)$-gain-sparse (respectively, $(k,m)$-gain-tight) and every balanced subgraph is $(k,\ell)$-gain-sparse.

For a group $\Gamma \leq \Iso (\mathbb{R}^d)$,
let $\Iso (\Gamma)$ be the group of isometries that preserve $\Gamma$-symmetry;
i.e.~isometries $g:\mathbb{R}^d \rightarrow \mathbb{R}^d$ where for any $\Gamma$-symmetric framework $(\mathbf{G},\mathbf{p})$,
the framework $(\mathbf{G},g \circ \mathbf{p})$ is also a $\Gamma$-symmetric framework.
We define $k(\Gamma) := \dim \Iso(\Gamma)$.
It can be shown that $k(\Gamma)$ is exactly the dimension of the space of trivial infinitesimal flexes of a $\Gamma$-symmetric framework with an orbit placement whose vertices affinely span $\mathbb{R}^d$.

\begin{example}
    If $\Gamma$ is a group of rotations in the plane,
    then $k(\Gamma)=1$.
    Likewise,
    if $\Gamma$ is a group generated by a reflection in a given line,
    then $k(\Gamma)=1$ also.
    However,
    if $\Gamma$ is generated by a reflection and at least one rotation,
    then $k(\Gamma)=0$.
    See \Cref{l:d2k} for more details.
\end{example}

We now finish this section by describing a necessary condition for a $\Gamma$-symmetric graph to be $\Gamma$-symmetrically rigid.

\begin{proposition}\label{prop:necgain}
    For a group $\Gamma \leq \Iso (\mathbb{R}^d)$,
    let $(G,\phi)$ be a $\Gamma$-symmetric gain graph with at least $d$ vertices.
    Fix $D = \binom{d+1}{2}$ and $k = k(\Gamma)$.
    If $(G,\phi)$ is $\Gamma$-symmetrically rigid and $|V| \geq d$,
    then it contains a $(d, D, k)$-gain-tight spanning subgraph.
\end{proposition}

\begin{proof}
    Choose a regular orbit placement $p$ of $(G,\phi)$ so that every subset of $t \leq d+1$ vertices is affinely independent\footnote{This is sometimes referred to as $p$ being in general position.}.
    By our choice of orbit placement,
    we have that the nullity of $R(G,\phi,p)$ is $k$.
    We may assume that the orbit rigidity matrix $R(G,\phi,p)$ has independent rows (and hence $|E| = \rank R(G,\phi,p)$),
    as deleting edges of $G$ that correspond to dependent rows in $R(G,\phi,p)$ does not affect the rank or nullity of $R(G,\phi,p)$.
    By the rank-nullity theorem,
    $|E| = \rank R(G,\phi,p) = d|V| - k$.
    
    Choose a subset $X \subset V$ with $|X| \geq d$,
    let $G' = (X,i(X))$, and let $p'$ and $\phi'$ be the restrictions of $p$ and $\phi$ to $H$.
    As the rows of $R(G',\phi',p') \times \mathbf{0}_{i(X) \times |V\setminus X|}$ are a subset of the rows of $R(G,\phi,p)$,
    the rows of $R(G',\phi',p')$ are independent.
    Since the nullity of $R(G',\phi',p')$ must be at least the dimension of the space of trivial infinitesimal flexes of $(G',\phi',p')$ (i.e.~$k$),
    we have $i(X) = \rank R(G,\phi,p) \leq d|X| - k$ by the rank-nullity theorem.
    
    Now suppose that the subgraph $G'$ is balanced.
    We first note that $G'$ cannot have loops or parallel edges as it is balanced.
    Choose a vertex $v_0 \in X$ and for every other vertex $v \in X$ choose a directed path $P_v = (e_1^v,\ldots,e^v_n)$ from $v_0$ to $v$.
    Now define $q$ to be the orbit placement of $G'$ where $q(v) = \phi(e_n^v) \ldots \phi(e_1^v) p(v)$ for every vertex $v \in X$,
    and define $\phi_q$ to be the gain map of $G'$ that assigns only trivial gains to edges.
    Since $X$ is balanced,
    we have that $R(G',\phi_q,q) = R(G',\phi_q,q)$.
    However,
    $R(G',\phi_q,q)$ is exactly the rigidity matrix of the pair $(G',q)$ when they are considered as defining a framework,
    thus the nullity of $R(G',\phi',p')$ is at least $D$.
    We now have $i(X) = \rank R(G,\phi,p) \leq d|X| - D$ by the rank-nullity theorem.
\end{proof}

\section{Cyclic plane symmetry groups}\label{sec:plane-1iso}

In the next two sections we consider symmetric frameworks in the plane for all possible symmetry groups $\Gamma$. We emphasise that this includes infinite point groups, such as any group generated by a rotation of irrational degree. The study of such groups seems to be new in rigidity theory.
It is well known that there is at most a 1-dimensional space of isometries that preserve $\Gamma$-symmetry. Since the only subgroups of $O(2)$ are the $n$-fold rotation groups,
	the two element groups generated by a single reflection or the groups generated by an $n$-fold rotation group and a single reflection, we have the following.

\begin{lemma}\label{l:d2k}
	Let $\Gamma \leq O(2)$ with $n := |\Gamma| \geq 2$ and $k := k(\Gamma)$.
	Then $k \in \{0,1\}$ and the following holds;
	\begin{enumerate}[(i)]
		\item if $k=1$ and $n=2$ then either $\Gamma$ is the $2$-fold rotation group or $\Gamma$ is generated by a single reflection,
		\item if $k=1$ and $3 \leq n < \infty$ then $\Gamma$ is the $n$-fold rotation group,
		\item if $k = 1$ and $n = \infty$ then $\Gamma$ is an infinite rotational group,
		i.e., $\Gamma \leq SO(2)$,
		\item if $k=0$ and $n < \infty$ then $n$ is even and $\Gamma$ is generated by $\frac{n}{2}$-fold rotation group and a single reflection, and
		\item if $k=0$ and $n = \infty$ then $\Gamma$ is generated by an infinite rotational group and a single reflection.
	\end{enumerate}
\end{lemma}


We first focus on the 1-dimensional case and then analyse symmetry groups with no continuous isometries in \Cref{sec:plane-noiso}.
We split this section into two parts depending on the number of elements in the group.

\subsection{Symmetry groups with at least 3 elements}
In this section we consider symmetry groups with at least three elements.

Let $G = (V,E)$ and $G'=(V',E')$ be multigraphs with $V' = V + v_0$. We define the following graph extensions which are used throughout the paper.
See \Cref{fig:extensions1} for illustrations.
\begin{itemize}
    \item We say $G'$ is formed from $G$ by a \emph{$0$-extension} if $E' = E + \{e'_1,e'_2\}$, where the edges $e'_1,e'_2$ have endpoints $v_0,v_1$ and $v_0,v_2$ respectively for some $v_1,v_2 \in V$ (where $v_1$ and $v_2$ may be the same vertex).
    \item We say $G'$ is formed from $G$ by a \emph{$1$-extension} if $E' = E -e + \{e'_1,e'_2,e'_3\}$, where each edge $e'_i$ has endpoints $v_0,v_i$ and $e \in E$ has endpoints $v_1,v_2$ for some $v_1,v_2,v_3 \in V$ (where $v_1$ and $v_2$ may be the same, and $v_3$ may be the same as one or both of $v_1,v_2$).
    \item We say $G'$ is formed from $G$ by a \emph{loop-1-extension} if $E' = E + \{ \ell, e'\}$,
    where $e'$ has endpoints $v_0,v$ for some $v \in V$ and $\ell$ is a loop at $v_0$.
\end{itemize}
\begin{figure}[ht]
    \centering
    \begin{tikzpicture}[every loop/.style={min distance=6mm,looseness=5},eframe/.style={black!30!white,rounded corners},earrow/.style={line width=2pt,-latex}]
        \node[eframe] (e0) at (6,-1) {0-extensions};
        \draw[eframe] (e0) -- ++(7.5,0) -- ++ (0,2.6) -- ++(-15,0) -- ++(0,-2.6) -- (e0);
        \node[eframe] (e1) at (6,-6.5) {1-extensions};
        \draw[eframe] (e1) -- ++(7.5,0) -- ++ (0,5.1) -- ++(-15,0) -- ++(0,-5.1) -- (e1);
        \node[eframe] (l1) at (6,-9.5) {loop-1-extension};
        \draw[eframe] (l1) -- ++(7.5,0) -- ++ (0,2.6) -- ++(-15,0) -- ++(0,-2.6) -- (l1);
        \begin{scope}
            \begin{scope}
                \draw[dashed] (0,0) circle [x radius=1.2cm,y radius=0.5cm];
                \node[vertex,label={[labelsty]below:$v_1$}] (1) at (-0.5,0.2) {};
                \node[vertex,label={[labelsty]below:$v_2$}] (2) at (0.5,0.2) {};
            \end{scope}
            \draw[earrow] (1.5,0)--(2.5,0);
            \begin{scope}[xshift=4cm]
                \draw[dashed] (0,0) circle [x radius=1.2cm,y radius=0.5cm];
                \node[vertex,label={[labelsty]below:$v_1$}] (1) at (-0.5,0.2) {};
                \node[vertex,label={[labelsty]below:$v_2$}] (2) at (0.5,0.2) {};
                \node[vertex,label={[labelsty]above:$v_0$}] (0) at (0,1) {};
                \draw[nedge] (1)edge(0) (2)edge(0);
            \end{scope}
        \end{scope}
        \begin{scope}[xshift=8cm]
            \begin{scope}
                \draw[dashed] (0,0) circle [x radius=1.2cm,y radius=0.5cm];
                \node[vertex,label={[labelsty]below:$v_1=v_2$}] (1) at (0,0.2) {};
            \end{scope}
            \draw[earrow] (1.5,0)--(2.5,0);
            \begin{scope}[xshift=4cm]
                \draw[dashed] (0,0) circle [x radius=1.2cm,y radius=0.5cm];
                \node[vertex,label={[labelsty]below:$v_1=v_2$}] (1) at (0,0.2) {};
                \node[vertex,label={[labelsty]above:$v_0$}] (0) at (0,1) {};
                \draw[nedge] (1)to[bend right=15](0);
                \draw[nedge] (1)to[bend right=-15](0);
            \end{scope}
        \end{scope}
        \begin{scope}[yshift=-3cm]
            \begin{scope}
                \begin{scope}
                    \draw[dashed] (0,0) circle [x radius=1.2cm,y radius=0.5cm];
                    \node[vertex,label={[labelsty]below:$v_1$}] (1) at (-0.6,0.2) {};
                    \node[vertex,label={[labelsty]below:$v_2$}] (2) at (0,0.2) {};
                    \node[vertex,label={[labelsty]below:$v_3$}] (3) at (0.6,0.2) {};
                    \draw[oedge] (1)edge(2);
                \end{scope}
                \draw[earrow] (1.5,0)--(2.5,0);
                \begin{scope}[xshift=4cm]
                    \draw[dashed] (0,0) circle [x radius=1.2cm,y radius=0.5cm];
                    \node[vertex,label={[labelsty]below:$v_1$}] (1) at (-0.6,0.2) {};
                    \node[vertex,label={[labelsty]below:$v_2$}] (2) at (0,0.2) {};
                    \node[vertex,label={[labelsty]below:$v_2$}] (3) at (0.6,0.2) {};
                    \node[vertex,label={[labelsty]above:$v_0$}] (0) at (0,1) {};
                    \draw[nedge] (1)edge(0) (2)edge(0) (3)edge(0);
                \end{scope}
            \end{scope}
            \begin{scope}[xshift=8cm]
                \begin{scope}
                    \draw[dashed] (0,0) circle [x radius=1.2cm,y radius=0.5cm];
                    \node[vertex,label={[labelsty]below:$v_1$}] (1) at (-0.6,0.2) {};
                    \node[vertex,label={[labelsty]below:$v_2=v_3$}] (2) at (0.4,0.2) {};
                    \draw[oedge] (1)edge(2);
                \end{scope}
                \draw[earrow] (1.5,0)--(2.5,0);
                \begin{scope}[xshift=4cm]
                    \draw[dashed] (0,0) circle [x radius=1.2cm,y radius=0.5cm];
                    \node[vertex,label={[labelsty]below:$v_1$}] (1) at (-0.6,0.2) {};
                    \node[vertex,label={[labelsty]below:$v_2=v_3$}] (2) at (0.4,0.2) {};
                    \node[vertex,label={[labelsty]above:$v_0$}] (0) at (0,1) {};
                    \draw[nedge] (1)edge(0) (2)to[bend right=15](0) (2)to[bend right=-15](0);
                \end{scope}
            \end{scope}
            
            \begin{scope}[yshift=-2.5cm]
                \begin{scope}
                    \draw[dashed] (0,0) circle [x radius=1.2cm,y radius=0.5cm];
                    \node[vertex,label={[labelsty]below:$v_1=v_2$}] (1) at (-0.3,0.2) {};
                    \node[vertex,label={[labelsty]below:$v_3$}] (3) at (0.6,0.2) {};
                    \draw[oedge] (1)to[in=160,out=200,loop](1);
                \end{scope}
                \draw[earrow] (1.5,0)--(2.5,0);
                \begin{scope}[xshift=4cm]
                    \draw[dashed] (0,0) circle [x radius=1.2cm,y radius=0.5cm];
                    \node[vertex,label={[labelsty]below:$v_1=v_2$}] (1) at (-0.3,0.2) {};
                    \node[vertex,label={[labelsty]below:$v_2$}] (3) at (0.5,0.2) {};
                    \node[vertex,label={[labelsty]above:$v_0$}] (0) at (0,1) {};
                    \draw[nedge] (1)to[bend right=15](0) (1)to[bend right=-15](0) (3)edge(0);
                \end{scope}
            \end{scope}
            \begin{scope}[xshift=8cm,yshift=-2.5cm]
                \begin{scope}
                    \draw[dashed] (0,0) circle [x radius=1.2cm,y radius=0.5cm];
                    \node[vertex,label={[labelsty]below:$v_1=v_2=v_3$}] (1) at (0,0.2) {};
                    \draw[oedge] (1)to[in=160,out=200,loop](1);
                \end{scope}
                \draw[earrow] (1.5,0)--(2.5,0);
                \begin{scope}[xshift=4cm]
                    \draw[dashed] (0,0) circle [x radius=1.2cm,y radius=0.5cm];
                    \node[vertex,label={[labelsty]below:$v_1=v_2=v_3$}] (1) at (0,0.2) {};
                    \node[vertex,label={[labelsty]above:$v_0$}] (0) at (0,1) {};
                    \draw[nedge] (1)edge(0) (1)to[bend right=30](0) (1)to[bend right=-30](0);
                \end{scope}
            \end{scope}
        \end{scope}
        
        \begin{scope}[yshift=-8.5cm,xshift=4cm]
            \begin{scope}
                \draw[dashed] (0,0) circle [x radius=1.2cm,y radius=0.5cm];
                \node[vertex,label={[labelsty]below:$v$}] (1) at (0,0) {};
            \end{scope}
            \draw[earrow] (1.5,0)--(2.5,0);
            \begin{scope}[xshift=4cm]
                \draw[dashed] (0,0) circle [x radius=1.2cm,y radius=0.5cm];
                \node[vertex,label={[labelsty]below:$v$}] (1) at (0,0) {};
                \node[vertex,label={[labelsty]left:$v_0$}] (0) at (0,1) {};
                \draw[nedge] (1)to(0);
                \draw[nedge] (0)to[in=60,out=120,loop](0);
            \end{scope}
        \end{scope}
    \end{tikzpicture}
    \caption{Extensions of multigraphs.}
    \label{fig:extensions1}
\end{figure}
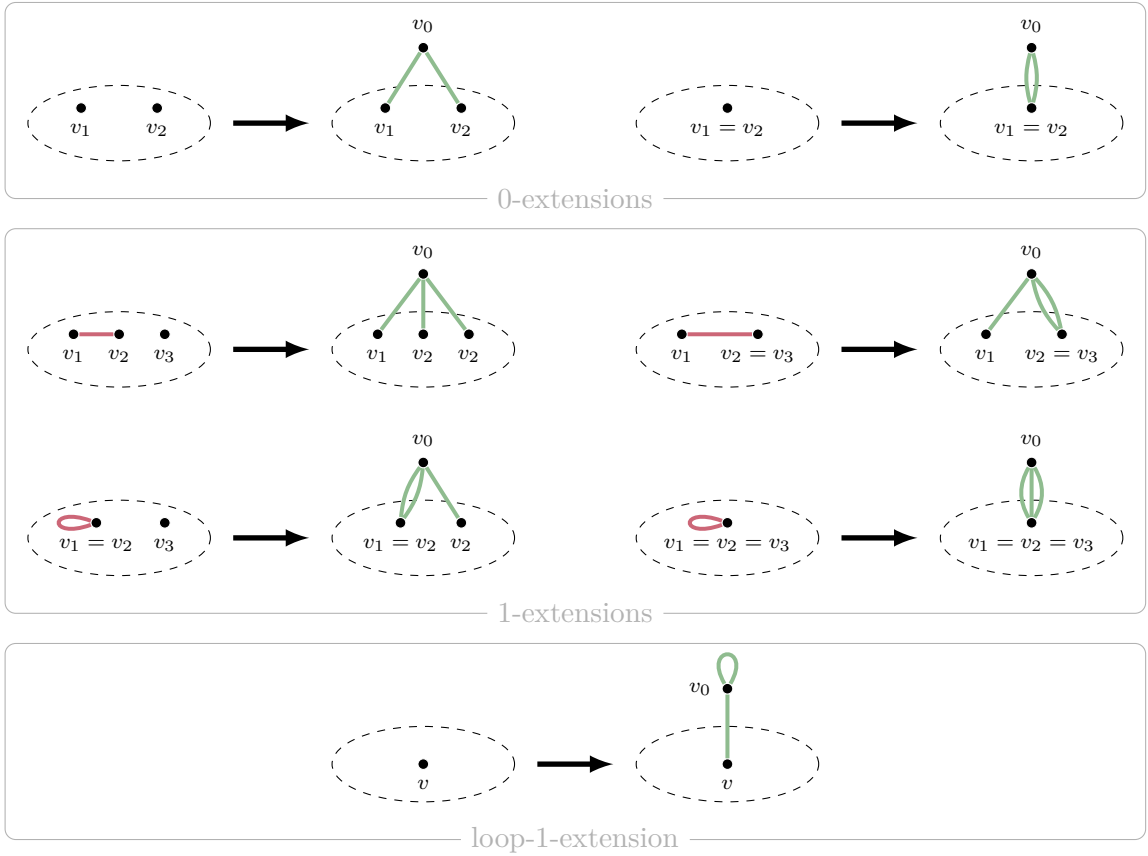
The inverse of a graph extension is known as a \emph{reduction},
e.g., if $G'$ is formed from a multigraph $G$ by a 0-/1-/loop-1-extension,
we say that $G$ is formed from $G'$ by a 0-/1-/loop-1-reduction.

We utilise the following special case of \cite[Theorem~1.6]{feketeszego06}. Define $K_1^j$ to be the multigraph comprised of a single vertex incident to $j$ loops.

\begin{theorem}\label{t:ext}
	A multigraph  $G$ is $(2,1)$-tight if and only if $G$ can be constructed from $K_1^1$ by a sequence of $0$-extensions, $1$-extensions and loop-$1$-extensions.
\end{theorem}


Let $\Gamma \leq O(2)$ and let $(G,\phi)$ be a $\Gamma$-gain graph.
Let $(G',\phi')$ be a $\Gamma$-gain graph where $G'$ is formed from $G$ by either a 0-extension, 1-extension or loop-1-extension.
We say that $(G',\phi')$ is formed from $(G,\phi)$ by a \emph{gained 0-extension} (respectively, \emph{gained 1-extension}, \emph{gained loop-1-extension}) if $\phi'(e,v,w) = \phi(e,v,w)$ for all $(e,v,w) \in \vec{E} \cap \vec{E}'$.
If $(G',\phi')$ is formed from $(G,\phi)$ by a gained 1-extension then we also require that $\phi'(e'_2,v_0,v_2)\phi(e'_1,v_1,v_0) = \phi(e,v_1,v_2)$ (with $e,e'_1,e'_2,v_0,v_1,v_2$ as given in the definition of a 1-extension).

\begin{lemma}\label{l:jordkasztani16}
	Let $(G,\phi)$ be a $\Gamma$-gain graph for a finite or infinite group $\Gamma \leq O(2)$.
	Let $(G',\phi')$ be a $\Gamma$-gain graph formed from $(G,\phi)$ by a gained 0-extension, gained 1-extension or gained loop-1-extension.
	If $(G,\phi)$ is $\Gamma$-symmetrically rigid in $\mathbb{R}^2$ then $(G',\phi')$ is $\Gamma$-symmetrically rigid in $\mathbb{R}^2$ also.
\end{lemma}

\begin{proof}
The lemma is proved in \cite[Lemma 6.1]{jordkasztani16} for finite groups. Their method is based on manipulating the orbit rigidity matrix which indeed applies independently of whether the group elements occurring in the matrix are elements of a finite group or an infinite group.
\end{proof}

It was shown in \cite{jordkasztani16} that for any group $\Gamma$ with $k(\Gamma)=1$, a $\Gamma$-gain graph is $(2,3,1)$-gain-tight if and only if it can be constructed from a vertex with a single loop by a sequence of gained 0-extensions, 1-extensions, and loop-1-extensions. Further, the analogues of \Cref{prop:necgain} and \Cref{l:jordkasztani16} in \cite{jordkasztani16} were only stated for finite point groups. Since our results apply to infinite point groups, we can very easily characterise exactly which $\Gamma$-symmetric gain graphs are $\Gamma$-rigid when $k(\Gamma)=1$. To be exact: given a finite or infinite point group $\Gamma \leq O(2)$ with $k(\Gamma)=1$, any $\Gamma$-gain graph is minimally $\Gamma$-rigid if and only if it is $(2,3,1)$-gain-tight.

\begin{lemma}\label{l:gainext}
	Let $(G,\phi)$ be a $\Gamma$-gain graph where $\Gamma \leq O(2)$ is a subgroup with $|\Gamma| \geq 3$,
	and let $G'=(V',E')$ be a multigraph formed from $G$ by either a $0$-extension, $1$-extension or loop-1-extension.
	If $(G,\phi)$ is $\Gamma$-symmetrically rigid in $\mathbb{R}^2$, then there exists a gain map $\phi' :\vec{E'} \rightarrow \Gamma$ where $\phi'(e,v,w)=\phi(e,v,w)$ for all $(e,v,w) \in \vec{E} \cap \vec{E}'$ and $(G',\phi')$ is $\Gamma$-symmetrically rigid in $\mathbb{R}^2$.
\end{lemma}

\begin{proof}
	First suppose that $G'$ is formed from $G$ by a $0$-extension.
	Let $e_1,e_2$ be the added edges with ends $\{v_0,v_1\},\{v_0,v_2\}$ respectively,
	where $v_1,v_2\in V$.
	If $v_1 \neq v_2$ we define $\phi'$ to be the gain map with $\phi'(e,v,w)=\phi(e,v,w)$ for all $e \in \vec{E}$ and $\phi'(e_1,v_0,v_1) = \phi'(e_2,v_0,v_2) = 1$,
	while if $v_1 = v_2$ we define $\phi'$ to be the gain map with $\phi'(e,v,w)=\phi(e,v,w)$ for all $e \in \vec{E}$, $\phi'(e_1,v_0,v_1) = 1$ and $\phi'(e_2,v_0,v_2) = \gamma$ for some non-trivial $\gamma \in \Gamma$.
	By \Cref{l:jordkasztani16},
	$(G',\phi')$ is $\Gamma$-symmetrically rigid in $\mathbb{R}^2$.
	
	Next assume $G'$ is formed from $G$ by a $1$-extension.
	Let $e_1,e_2,e_3$ be the added edges with ends $\{v_0,v_1\}$, $\{v_0,v_2\}$, $\{v_0,v_3\}$ respectively;
	by relabelling if required we suppose the (possibly equal) vertices $v_1,v_2 \in V$ are the ends of the removed edge $e$,
	and if $v_3 \in \{v_1,v_2\}$ then $v_3 = v_1$.
	Choose $\gamma \in \Gamma$ such that $\gamma \neq 1$;
	if $v_1=v_2=v_3$ then we also require that $\gamma \neq \phi(e,v_1,v_2)$,
	which we can guarantee as $|\Gamma| \geq 3$.
	Define $\phi'$ to be the gain map with $\phi'(e,v,w)=\phi(e,v,w)$ for all $(e,v,w) \in \vec{E} \cap \vec{E}'$, $\phi'(e_1,v_0,v_1) = 1$,
	$\phi'(e_2,v_0,v_2) = \phi(e,v_1,v_2)$ and $\phi'(e_3,v_0,v_3) = \gamma$.
	By \Cref{l:jordkasztani16},
	$(G',\phi')$ is $\Gamma$-symmetrically rigid in $\mathbb{R}^2$.

	Finally suppose $G'$ is formed from $G$ by a loop-$1$-extension.
	Let $\ell$ be the added loop at $v_0$ and $e_1$ the added edge with ends $\{v_0,v_1\}$.
	Choose any non-trivial $\gamma \in \Gamma$.
	Define $\phi'$ to be the gain map with $\phi'(e,v,w)=\phi(e,v,w)$ for all $(e,v,w) \in \vec{E}$, $\phi'(e_1,v_0,v_1) = 1$ and $\phi'(\ell,v_0,v_0) = \gamma$.
	By \Cref{l:jordkasztani16},
	$(G',\phi')$ is $\Gamma$-symmetrically rigid in $\mathbb{R}^2$.
\end{proof}

\begin{theorem}\label{t:d2k1}
	Let $\Gamma$ be a rotational subgroup of $O(2)$ with $|\Gamma| \geq 3$,
	and let $G$ be a $(2,1)$-tight multigraph.
	Then there exists a gain map $\phi : \vec{E} \rightarrow \Gamma$ such that $(G, \phi)$ is $\Gamma$-symmetrically rigid in $\mathbb{R}^2$.
\end{theorem}

\begin{proof}
 	By \Cref{t:ext},
	$G$ can be formed from $K_1^1$ by a sequence of $0$-extensions, $1$-extensions and loop-1-extensions.
	Choose any non-identity element $\gamma \in \Gamma$ and define $\phi': \vec{E(K_1^1)} \rightarrow \Gamma$ to be the gain map that maps the single loop of $K_1^1$ to $\gamma$.
	It is immediate that $(K_1^1,\phi')$ is $\Gamma$-symmetrically rigid in $\mathbb{R}^2$.
	By applying \Cref{l:gainext} inductively we see that there exists a gain map $\phi: \vec{E} \rightarrow \Gamma$ so that $(G,\phi)$ is $\Gamma$-symmetrically rigid in $\mathbb{R}^2$.	
\end{proof}

\subsection{Groups with 2 elements}

Unfortunately \Cref{t:d2k1} does not hold when $|\Gamma|=2$.
For example, take the multigraph with two vertices and three non-loop edges. 
Any gain map would require two edges to have identical gains, so no gain map is possible.
To avoid this problem, we next refine \Cref{t:ext}, using the same operations, to the more restrictive class of $(2,1)$-tight multigraphs with no triple of parallel edges\footnote{This is a smaller class, so certainly every multigraph in it can be constructed, as guaranteed by \Cref{t:ext}, however for our inductive rigidity application we must refine the construction so that no intermediate multigraph in the construction contains a triple of parallel edges.}. This result is then deployed to show that triples of parallel edges are the only block stopping \Cref{t:d2k1} from holding when $|\Gamma|=2$.


\begin{lemma}\label{l:extnotriple}
	Let $G=(V,E)$ be a multigraph with no triple of parallel edges.
	Then $G$ is $(2,1)$-tight if and only if $G$ can be constructed from $K_1^1$ by a sequence of $0$-extensions, $1$-extensions and loop-$1$-extensions that do not form a triple of parallel edges.
\end{lemma}

\begin{proof}
One direction is trivial. For the converse suppose $G=(V,E)$ is a $(2,1)$-tight multigraph, distinct from $K_1^1$, with no triple of parallel edges. Since $|E|=2|V|-1$, $G$ has a vertex of degree at most 3. If there exists a degree 2 vertex the multigraph can be reduced to a smaller one with the required properties by a $0$-reduction. So we may suppose $v\in V$ has degree 3. If $v$ is incident to a loop then a loop-1-reduction is always possible. So $v$ is not incident to a loop. It follows from the proof of \Cref{t:ext} (see \cite{feketeszego06}) that there is a 1-reduction at $v$ to a smaller $(2,1)$-tight multigraph. Suppose that every such reduction at $v$ creates a triple of parallel edges. Note that two of the possible 1-reductions add only a new loop so cannot create a triple of parallel edges.

If $N(v)=\{x,y,z\}$ then without loss we may assume that $G$ contains two parallel edges between $x$ and $y$. We now consider the 1-reduction at $v$ that adds $xz$. This fails to result in a $(2,1)$-tight multigraph if and only if there is a $(2,1)$-tight subgraph $H$ of $G$ such that $x,z\in V(H)$ and $y,v\notin V(H)$. (To see that $y\notin V(H)$, simply note that if it was then $H+v$ would violate $(2,1)$-sparsity.)
However, if such a subgraph $H$ exists then we may add $y$, its two parallel edges to $x$, $v$ and its three edges to obtain a subgraph of $G$ that is not $(2,1)$-sparse.
Hence, the 1-reduction at $v$ adding $xz$ results in a $(2,1)$-tight multigraph, so there must exist two parallel edges between $x$ and $z$. By a similar argument, there must also exist two parallel edges between $y$ and $z$. Thus the subgraph of $G$ induced by $v$ and its neighbours has 9 edges and 4 vertices, contradicting $(2,1)$-sparsity.

Therefore, $v$ has two distinct neighbours $x,y$ with a double edge from $v$ to $x$.
We now know two things from our assumptions:
(i) the vertex $x$ lies in a $(2,1)$-tight subgraph $H$ of $G$ (as otherwise we could perform the 1-reduction at $v$ that adds a loop at $x$ which does not form a triple of parallel edges), and
(ii) there exists a pair of parallel edges between $x,y$ (since the only remaining possible 1-reduction at $v$ to a smaller $(2,1)$-tight multigraph adds an edge between $x,y$, and this must form a triple of parallel edges).
However, the subgraph of $G$ formed from adding the vertices $y,v$ and the edges between the vertices $x,y,v$ to $H$ is not $(2,1)$-tight,
contradicting our original assumption. 
This completes the proof.
\end{proof}

\begin{lemma}\label{l:gainextnotriple}
	Let $(G,\phi)$ be a $\Gamma$-gain graph where $\Gamma$ is a cyclic group,
	and let $G'=(V',E')$ be a multigraph formed from $G$ by either a $0$-extension, $1$-extension or loop-$1$-extension.
	Suppose neither $G$ or $G'$ contain a triple of parallel edges.
	If $(G,\phi)$ is $\Gamma$-symmetrically rigid in $\mathbb{R}^2$ then there exists a gain map $\phi' :\vec{E'} \rightarrow \Gamma$ where $\phi'(e,v,w)=\phi(e,v,w)$ for all $(e,v,w) \in \vec{E} \cap \vec{E}'$ and $(G',\phi')$ is $\Gamma$-symmetrically rigid in $\mathbb{R}^2$.
\end{lemma}

\begin{proof}
    This follows from \Cref{l:gainext} by noting that the only extension moves that require $|\Gamma| \geq 3$ are those that create a triple of parallel edges.
\end{proof}

\begin{theorem}
	Let $\Gamma$ be a cyclic subgroup of $O(2)$ with $|\Gamma| =2$,
	and let $G$ be a $(2,1)$-tight multigraph with no triple of parallel edges.
	Then there exists a gain map $\phi : \vec{E} \rightarrow \Gamma$ such that $(G, \phi)$ is $\Gamma$-symmetrically rigid in $\mathbb{R}^2$.
\end{theorem}

\begin{proof}
	By \Cref{l:extnotriple},
	$G$ can be formed from $K_1^1$ by a sequence of $0$-extensions, $1$-extensions and loop-1-extensions that do not form a triple of parallel edges and edge joining.
	Choose the non-identity element $\gamma \in \Gamma$ and define $\phi': \vec{E(K_1^1)} \rightarrow \Gamma$ to be the gain map that maps the single loop of $K_1^1$ to $\gamma$.
	It is immediate that $(K_1^1,\phi')$ is $\Gamma$-symmetrically rigid in $\mathbb{R}^2$.
	By applying \Cref{l:gainextnotriple} inductively we see that there exists a gain map $\phi: \vec{E} \rightarrow \Gamma$ so that $(G,\phi)$ is $\Gamma$-symmetrically rigid in $\mathbb{R}^2$.
\end{proof}

\section{Dihedral plane symmetry groups}\label{sec:plane-noiso}

We now extend the results of the previous section to apply when the symmetry group, whether finite or infinite, contains both rotations and reflections.

\subsection{Groups with at least 6 elements}

We have defined $0$-extensions, $1$-extensions and loop-$1$-extensions earlier.
In the following we need additional extension constructions.
Let $G = (V,E)$ and $G'=(V',E')$ be multigraphs with $V' = V + v_0$.
See \Cref{fig:extensions2} for visualisations.
\begin{itemize}
    \item We say $G'$ is formed from $G$ by a \emph{$2$-extension} if $E' = E -\{e_1 , e_2\} + \{e'_1,e'_2,e'_3,e'_4\}$,
    where the edges $e_1,e_2 \in E$ have endpoints $v_1,v_2$ and $v_3,v_4$ respectively,
    and each edge $e'_i$ has endpoints $v_0,v_i$.
    Note that the $v_1,v_2,v_3,v_4$ do not need to be different.
    \item We say $G'$ is formed from $G$ by a \emph{loop-2-extension} if $E' = E - e + \{e'_1,e'_2, \ell\}$,
    where $e \in E$ is an edge with endpoints $v_1,v_2$, each edge $e'_i$ has endpoints $v_0,v_i$ and $\ell$ is a loop at $v_0$. Note that $v_1$ and $v_2$ do not need to be distinct.
    \item We say $G'$ is formed from $G$ by a \emph{loop-0-extension} if $E' = E \{\ell_1,\ell_2\}$,
    where $\ell_1,\ell_2$ are loops at $v_0$.
\end{itemize}

\begin{figure}[ht]
    \centering
    \begin{tikzpicture}[every loop/.style={min distance=6mm,looseness=5},eframe/.style={black!30!white,rounded corners},earrow/.style={line width=2pt,-latex}]
        \node[eframe] (e2) at (6,-8.5) {2-extensions};
        \draw[eframe] (e2) -- ++(7.5,0) -- ++ (0,10.1) -- ++(-15,0) -- ++(0,-10.1) -- (e2);
        \node[eframe] (l2) at (6,-11.5) {loop-2-extension};
        \draw[eframe] (l2) -- ++(7.5,0) -- ++ (0,2.6) -- ++(-15,0) -- ++(0,-2.6) -- (l2);
        
        \begin{scope}
            \begin{scope}
                \begin{scope}
                    \draw[dashed] (0,0) circle [x radius=1.2cm,y radius=0.5cm];
                    \node[vertex,label={[labelsty]below:$v_1$}] (1) at (-0.6,0.2) {};
                    \node[vertex,label={[labelsty]below:$v_2$}] (2) at (-0.2,0.2) {};
                    \node[vertex,label={[labelsty]below:$v_3$}] (3) at (0.2,0.2) {};
                    \node[vertex,label={[labelsty]below:$v_4$}] (4) at (0.6,0.2) {};
                    \draw[oedge] (1)edge(2);
                    \draw[oedge] (3)edge(4);
                \end{scope}
                \draw[earrow] (1.5,0)--(2.5,0);
                \begin{scope}[xshift=4cm]
                    \draw[dashed] (0,0) circle [x radius=1.2cm,y radius=0.5cm];
                    \node[vertex,label={[labelsty]below:$v_1$}] (1) at (-0.6,0.2) {};
                    \node[vertex,label={[labelsty]below:$v_2$}] (2) at (-0.2,0.2) {};
                    \node[vertex,label={[labelsty]below:$v_3$}] (3) at (0.2,0.2) {};
                    \node[vertex,label={[labelsty]below:$v_4$}] (4) at (0.6,0.2) {};
                    \node[vertex,label={[labelsty]above:$v_0$}] (0) at (0,1) {};
                    \draw[nedge] (1)edge(0) (2)edge(0) (3)edge(0) (4)edge(0);
                \end{scope}
            \end{scope}
            \begin{scope}[xshift=8cm]
                \begin{scope}
                    \draw[dashed] (0,0) circle [x radius=1.2cm,y radius=0.5cm];
                    \node[vertex,label={[labelsty,label distance=-3pt]-135:$v_1$}] (1) at (-0.6,0.2) {};
                    \node[vertex,label={[labelsty]below:$v_2=v_3$}] (2) at (0,0) {};
                    \node[vertex,label={[labelsty,label distance=-3pt]-45:$v_4$}] (4) at (0.6,0.2) {};
                    \draw[oedge] (1)edge(2);
                    \draw[oedge] (2)edge(4);
                \end{scope}
                \draw[earrow] (1.5,0)--(2.5,0);
                \begin{scope}[xshift=4cm]
                    \draw[dashed] (0,0) circle [x radius=1.2cm,y radius=0.5cm];
                    \node[vertex,label={[labelsty,label distance=-3pt]-135:$v_1$}] (1) at (-0.6,0.2) {};
                    \node[vertex,label={[labelsty]below:$v_2=v_3$}] (2) at (0,0) {};
                    \node[vertex,label={[labelsty,label distance=-3pt]-45:$v_4$}] (4) at (0.6,0.2) {};
                    \node[vertex,label={[labelsty]above:$v_0$}] (0) at (0,1) {};
                    \draw[nedge] (1)edge(0) (2)to[bend right=15](0) (2)to[bend right=-15](0) (4)edge(0);
                \end{scope}
            \end{scope}
            
            \begin{scope}[yshift=-2.5cm,xshift=4cm]
                \begin{scope}
                    \draw[dashed] (0,0) circle [x radius=1.2cm,y radius=0.5cm];
                    \node[vertex,label={[labelsty]below:$v_1=v_3$}] (1) at (-0.6,0.2) {};
                    \node[vertex,label={[labelsty]below:$v_2=v_4$}] (4) at (0.6,0.2) {};
                    \draw[oedge] (1)to[bend right=15](4);
                    \draw[oedge] (1)to[bend right=-15](4);
                \end{scope}
                \draw[earrow] (1.5,0)--(2.5,0);
                \begin{scope}[xshift=4cm]
                    \draw[dashed] (0,0) circle [x radius=1.2cm,y radius=0.5cm];
                    \node[vertex,label={[labelsty]below:$v_1=v_3$}] (1) at (-0.6,0.2) {};
                    \node[vertex,label={[labelsty]below:$v_2=v_4$}] (4) at (0.6,0.2) {};
                    \node[vertex,label={[labelsty]above:$v_0$}] (0) at (0,1) {};
                    \draw[nedge] (1)to[bend right=15](0) (1)to[bend right=-15](0) (4)to[bend right=15](0) (4)to[bend right=-15](0);
                \end{scope}
            \end{scope}
        \end{scope}
        \begin{scope}[yshift=-5cm]
            \begin{scope}
                \begin{scope}
                    \draw[dashed] (0,0) circle [x radius=1.2cm,y radius=0.5cm];
                    \node[vertex,label={[labelsty]below:$v_1=v_2$}] (1) at (-0.3,0) {};
                    \node[vertex,label={[labelsty]below:$v_3$}] (3) at (0.2,0.2) {};
                    \node[vertex,label={[labelsty]below:$v_4$}] (4) at (0.6,0.2) {};
                    \draw[oedge] (1)to[in=160,out=200,loop](1);
                    \draw[oedge] (3)edge(4);
                \end{scope}
                \draw[earrow] (1.5,0)--(2.5,0);
                \begin{scope}[xshift=4cm]
                    \draw[dashed] (0,0) circle [x radius=1.2cm,y radius=0.5cm];
                    \node[vertex,label={[labelsty]below:$v_1=v_2$}] (1) at (-0.3,0) {};
                    \node[vertex,label={[labelsty]below:$v_3$}] (3) at (0.2,0.2) {};
                    \node[vertex,label={[labelsty]below:$v_4$}] (4) at (0.6,0.2) {};
                    \node[vertex,label={[labelsty]above:$v_0$}] (0) at (0,1) {};
                    \draw[nedge] (1)to[bend right=15](0) (1)to[bend right=-15](0) (3)edge(0) (4)edge(0);
                \end{scope}
            \end{scope}
            \begin{scope}[xshift=8cm]
                \begin{scope}
                    \draw[dashed] (0,0) circle [x radius=1.2cm,y radius=0.5cm];
                    \node[vertex,label={[labelsty]below:$v_1=v_2=v_3$}] (1) at (-0.1,0.1) {};
                    \node[vertex,label={[labelsty]right:$v_4$}] (4) at (0.5,0.1) {};
                    \draw[oedge] (1)to[in=160,out=200,loop](1);
                    \draw[oedge] (1)edge(4);
                \end{scope}
                \draw[earrow] (1.5,0)--(2.5,0);
                \begin{scope}[xshift=4cm]
                    \draw[dashed] (0,0) circle [x radius=1.2cm,y radius=0.5cm];
                    \node[vertex,label={[labelsty]below:$v_1=v_2=v_3$}] (1) at (-0.1,0.1) {};
                    \node[vertex,label={[labelsty]right:$v_4$}] (4) at (0.5,0.1) {};
                    \node[vertex,label={[labelsty]above:$v_0$}] (0) at (0,1) {};
                    \draw[nedge] (1)to[bend right=20](0) (1)to[bend right=-20](0) (1)edge(0) (4)edge(0);
                \end{scope}
            \end{scope}
        \end{scope}
        \begin{scope}[yshift=-7.5cm]
            \begin{scope}
                \begin{scope}
                    \draw[dashed] (0,0) circle [x radius=1.2cm,y radius=0.5cm];
                    \node[vertex,label={[labelsty]below:$v_1=v_2$}] (1) at (-0.3,0) {};
                    \node[vertex,label={[labelsty]below:$v_3=v_4$}] (3) at (0.35,0.2) {};
                    \draw[oedge] (1)to[in=160,out=200,loop](1);
                    \draw[oedge] (3)to[in=20,out=-20,loop](3);
                \end{scope}
                \draw[earrow] (1.5,0)--(2.5,0);
                \begin{scope}[xshift=4cm]
                    \draw[dashed] (0,0) circle [x radius=1.2cm,y radius=0.5cm];
                    \node[vertex,label={[labelsty]below:$v_1=v_2$}] (1) at (-0.3,0) {};
                    \node[vertex,label={[labelsty]below:$v_3=v_4$}] (3) at (0.35,0.2) {};
                    \node[vertex,label={[labelsty]above:$v_0$}] (0) at (0,1) {};
                    \draw[nedge] (1)to[bend right=15](0) (1)to[bend right=-15](0) (3)to[bend right=15](0) (3)to[bend right=-15](0);
                \end{scope}
            \end{scope}
            \begin{scope}[xshift=8cm]
                \begin{scope}
                    \draw[dashed] (0,0) circle [x radius=1.2cm,y radius=0.5cm];
                    \node[vertex,label={[labelsty,align=center]below:$v_1=v_2=$\\[-1ex]$v_3=v_4$}] (1) at (0,0.25) {};
                    \draw[oedge] (1)to[in=160,out=200,loop](1);
                    \draw[oedge] (1)to[in=20,out=-20,loop](1);
                \end{scope}
                \draw[earrow] (1.5,0)--(2.5,0);
                \begin{scope}[xshift=4cm]
                    \draw[dashed] (0,0) circle [x radius=1.2cm,y radius=0.5cm];
                    \node[vertex,label={[labelsty,align=center]below:$v_1=v_2=$\\[-1ex]$v_3=v_4$}] (1) at (0,0.25) {};
                    \node[vertex,label={[labelsty]above:$v_0$}] (0) at (0,1) {};
                    \draw[nedge] (1)to[bend right=40](0) (1)to[bend right=-40](0) (1)to[bend right=15](0) (1)to[bend right=-15](0);
                \end{scope}
            \end{scope}
        \end{scope}
        
        \begin{scope}[yshift=-10.5cm]
            \begin{scope}
                \begin{scope}
                    \draw[dashed] (0,0) circle [x radius=1.2cm,y radius=0.5cm];
                    \node[vertex,label={[labelsty]below:$v_1$}] (1) at (-0.5,0.2) {};
                    \node[vertex,label={[labelsty]below:$v_2$}] (2) at (0.5,0.2) {};
                    \draw[oedge] (1)edge(2);
                \end{scope}
                \draw[earrow] (1.5,0)--(2.5,0);
                \begin{scope}[xshift=4cm]
                    \draw[dashed] (0,0) circle [x radius=1.2cm,y radius=0.5cm];
                    \node[vertex,label={[labelsty]below:$v_1$}] (1) at (-0.5,0.2) {};
                    \node[vertex,label={[labelsty]below:$v_2$}] (2) at (0.5,0.2) {};
                    \node[vertex,label={[labelsty]left:$v_0$}] (0) at (0,1) {};
                    \draw[nedge] (1)to(0) (2)edge(0);
                    \draw[nedge] (0)to[in=60,out=120,loop](0);
                \end{scope}
            \end{scope}
            \begin{scope}[xshift=8cm]
                \begin{scope}
                    \draw[dashed] (0,0) circle [x radius=1.2cm,y radius=0.5cm];
                    \node[vertex,label={[labelsty]below:$v_1=v_2$}] (1) at (0,0.2) {};
                    \draw[oedge] (1)to[in=160,out=200,loop](1);
                \end{scope}
                \draw[earrow] (1.5,0)--(2.5,0);
                \begin{scope}[xshift=4cm]
                    \draw[dashed] (0,0) circle [x radius=1.2cm,y radius=0.5cm];
                    \node[vertex,label={[labelsty]below:$v_1=v_2$}] (1) at (0,0.2) {};
                    \node[vertex,label={[labelsty]left:$v_0$}] (0) at (0,1) {};
                    \draw[nedge] (1)to[bend right=20](0) (1)to[bend right=-20](0);
                    \draw[nedge] (0)to[in=60,out=120,loop](0);
                \end{scope}
            \end{scope}
        \end{scope}
    \end{tikzpicture}
    \caption{Extensions of multigraphs.}
    \label{fig:extensions2}
\end{figure}
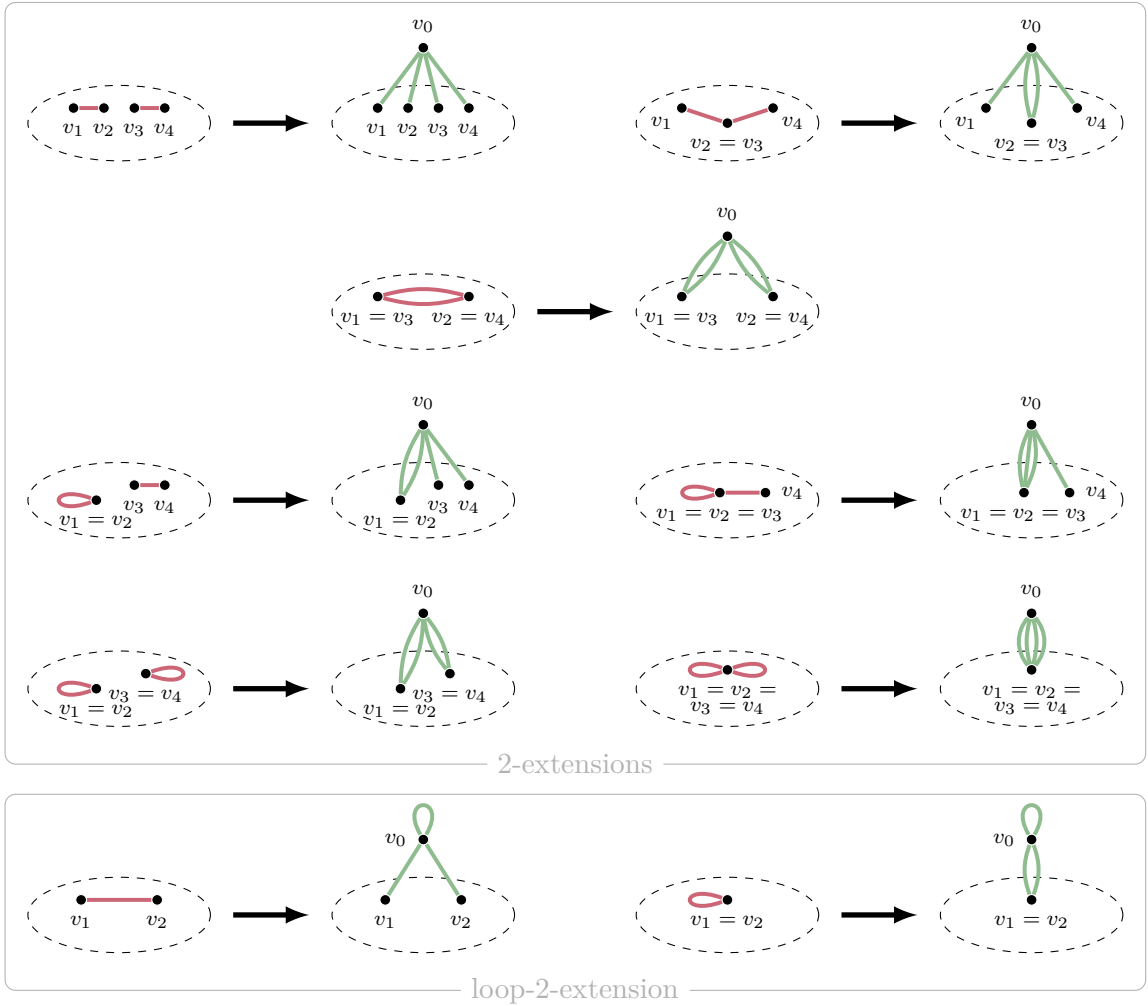

The following is another special case of \cite[Theorem~1.6]{feketeszego06}.
Since $(2,0)$-tight multigraphs can be disconnected we must add a loop-0-extension where the new vertex is incident to two loops (and is not connected to the original multigraph). Clearly this was unnecessary in the $(2,1)$-tight case.

\begin{theorem}\label{t:ext2}
	Let $G$ be a multigraph.
	Then $G$ is $(2,0)$-tight if and only if $G$ can be constructed from $K_1^2$ by a sequence of $0$-extensions, loop-0-extensions, $1$-extensions, loop-$1$-extensions, 2-extensions and loop-2-extensions.
\end{theorem}

Let $\Gamma \leq O(2)$ and let $(G,\phi)$ be a $\Gamma$-gain graph.
Let $(G',\phi')$ be a $\Gamma$-gain graph where $G'$ be formed from $G$ by either a 2-extension or a loop-2-extension.
We say that $(G',\phi')$ is formed from $(G,\phi)$ by a \emph{gained 2-extension} (respectively, \emph{gained loop-2-extension}, \emph{gained loop-0-extension}) if $\phi'(e,v,w) = \phi(e,v,w)$ for all $(e,v,w) \in \vec{E} \cap \vec{E}'$,
plus the following extra requirement:
\begin{itemize}
    \item If $(G',\phi')$ is formed from $(G,\phi)$ by a gained 2-extension then we additionally require that $\phi'(e'_2,v_0,v_2)\phi(e'_1,v_1,v_0) = \phi(e_1,v_1,v_2)$ and $\phi'(e'_4,v_0,v_4)\phi(e'_3,v_3,v_0) = \phi(e_2,v_3,v_4)$ (with $e_i$'s, $e'_i$'s, $v_i$'s as given in the definition of a 2-extension).
    \item If $(G',\phi')$ is formed from $(G,\phi)$ by a gained loop-2-extension then we also require that $\phi'(e'_2,v_0,v_2)\phi(e'_1,v_1,v_0) = \phi(e,v_1,v_2)$ (with $e,e'_1,e'_2,v_0,v_1,v_2$ as given in the definition of a loop-2-extension).
    \item If $(G',\phi')$ is formed from $(G,\phi)$ by a gained loop-0-extension then we also require that the generated subgroup $\langle \phi'(\ell_1,v_0,v_0),\phi(\ell_2,v_0,v_0) \rangle$ is dihedral.
\end{itemize}

Similarly to \Cref{l:jordkasztani16} we have the following extension lemma for both finite and infinite groups which was proved in the finite case in \cite[Lemmas 8.5 \& 8.7]{jordkasztani16}.

\begin{lemma}\label{l:jordkasztani16.2}
	Let $(G,\phi)$ be a $\Gamma$-gain graph for a finite or infinite group $\Gamma \leq O(2)$ that contains both rotations and reflections.
	Let $(G',\phi')$ be a $\Gamma$-gain graph formed from $(G,\phi)$ by a gained 2-extension or a gained loop-2-extension.
	If $(G,\phi)$ is $\Gamma$-symmetrically rigid in $\mathbb{R}^2$ then $(G',\phi')$ is $\Gamma$-symmetrically rigid in $\mathbb{R}^2$ also.
\end{lemma}

With this we can obtain the following result.

\begin{lemma}\label{l:gain2ext}
	Let $(G,\phi)$ be a $\Gamma$-gain graph where $\Gamma \leq O(2)$ is a subgroup with $|\Gamma| \geq 6$  that contains both rotations and reflections,
	and let $G'=(V',E')$ be a multigraph formed from $G$ by either a $2$-extension or a loop-$2$-extension.
	If $(G,\phi)$ is $\Gamma$-symmetrically rigid in $\mathbb{R}^2$ then there exists a gain map $\phi' :\vec{E'} \rightarrow \Gamma$, where $\phi'(e,v,w)=\phi(e,v,w)$ for all $(e,v,w) \in \vec{E} \cap \vec{E}'$ and $(G',\phi')$ is $\Gamma$-symmetrically rigid in $\mathbb{R}^2$.
\end{lemma}

\begin{proof}
	First suppose $G'$ is formed by a 2-extension.
	Let $\gamma_1 = \phi(e_1,v_1,v_2)$ and $\gamma_2 = \phi(e_2,v_3,v_4)$.
	We now need to show that we can choose pairwise-disjoint $\mu_1,\mu_2,\mu_3,\mu_4 \in \Gamma$ so that if we set $\phi'$ to be the gain map of $G'$ where $\phi'(e'_i,v_0,v_i)=\mu_i$ for each edge $e'_i$ and $\phi'(e,v,w)=\phi(e,v,w)$ otherwise, and where $\mu_2^{-1}\mu_1 = \gamma_1$ and $\mu_4^{-1}\mu_3 = \gamma_2$, then $\phi'$ is a gain map of $G'$;
	i.e.~if $e'_i$ and $e'_j$ are parallel in $G'$,
	then $\mu_i \neq \mu_j$.
	By switching $v_1,v_2$ and/or switching $v_3,v_4$ if required,
	we see that there are 5 possibilities to deal with.
	\begin{enumerate}
	\item $e_1,e_2$ share no endpoints:
	In this case neither of the edges $e'_1$ and $e'_2$ is parallel to either $e'_3$ or $e'_4$.
	Set $\mu_1 = \gamma_1$, $\mu_3 = \gamma_2$ and $\mu_2=\mu_4$.
	If $\gamma_1=1$ then $e'_1$ and $e'_2$ are parallel if and only if $e_1$ is a loop,
	which contradicts that no loop may have trivial gain.
	Hence $\phi'$ is a gain map of $G'$.
	
	\item $v_1 \neq v_2$, $v_1 \neq v_4$ and $v_2 = v_3$:
	Choose $\lambda \in \Gamma$ so that $\lambda \gamma_1 \neq 1$.
	We now set $\mu_1=\gamma_1$, $\mu_2=1$, $\mu_3 = \lambda \gamma_2$ and $\mu_4 = \lambda$ to obtain the desired gain map $\phi'$ of $G'$.
	
	\item $v_1 = v_3$, $v_2 = v_4$ and $v_1 \neq v_2$:
	By applying switching operations, we may assume that $\gamma_1=1$.
	Since $(G,\phi)$ is a gain graph, it follows that $\gamma_2 \neq 1$.
	As $|\Gamma| \geq 3$ and every element of a group has exactly one inverse,
	we can choose an element $\lambda \in \Gamma$ so that $\lambda \neq 1$ and $\lambda \gamma_2 \neq 1$.
	We now set $\mu_1=\mu_2=1$, $\mu_3 = \lambda \gamma_2$ and $\mu_4 = \lambda$ to obtain the desired gain map $\phi'$ of $G'$.
	
	\item $v_1 = v_2 = v_3$ and $v_1 \neq v_4$:
	As $e_1$ is a loop,
	we have $\gamma_1 \neq 1$.
	As $|\Gamma| \geq 3$,
	we can choose an element $\lambda \in \Gamma$ so that $\lambda \gamma_2 \notin \{1,\gamma_1\}$.
	We now set $\mu_1=\gamma_1$, $\mu_2=1$, $\mu_3 = \lambda \gamma_2$ and $\mu_4 = \lambda$ to obtain the desired gain map $\phi'$ of $G'$.
	
	\item $v_1 = v_2 = v_3 = v_4$:
	As $e_1,e_2$ are both loops at the same vertex, we have $\gamma_1,\gamma_2 \neq 1$ and $\gamma_1 \neq \gamma_2$.
	Since $|\Gamma| \geq 5$,
	there exists $\lambda \in \Gamma \setminus \{1,\gamma_1\}$ so that $\lambda \gamma_2 \notin \{ 1,\gamma_1\}$.
	We note that $\lambda \neq \lambda \gamma_2$ since $\gamma_2 \neq 1$.
	We now set $\mu_1=\gamma_1$, $\mu_2=1$, $\mu_3 = \lambda \gamma_2$ and $\mu_4 = \lambda$ to obtain the desired gain map $\phi'$ of $G'$.
	\end{enumerate}
	
	Now suppose $G'$ is formed by a loop-2-extension.
	Let $\gamma = \phi(e,v_1,v_2)$.
	We now need to show that we can choose pairwise-disjoint $\mu_1,\mu_2,\mu \in \Gamma$ so that if we set $\phi'$ to be the gain map of $G'$, where $\phi'(e'_i,v_0,v_i)=\mu_i$ for each edge $e'_i$, $\phi'(\ell,v_0,v_0) = \mu$ and $\phi'(e,v,w)=\phi(e,v,w)$ otherwise, and where $\mu_2^{-1}\mu_1 = \gamma$ and $\mu \neq 1$, then $\phi'$ is a gain map of $G'$;
	i.e.~if $e'_1$ and $e'_2$ are parallel in $G'$,
	then $\mu_1 \neq \mu_2$.
	If $e'_1,e'_2$ are not parallel,
	any $\mu_1,\mu_2,\mu \in \Gamma$ where $\mu_2^{-1}\mu_1 = \gamma$ and $\mu \neq 1$ suffice.
	Suppose that $e'_1,e'_2$ are parallel.
	Then $e$ is a loop and $\gamma \neq 1$.
	We now set $\mu_1=\gamma_1$, $\mu_2=1$ and choose any $\mu \in \Gamma \setminus \{1\}$ to obtain the desired gain map $\phi'$ of $G'$.
\end{proof}

\begin{theorem}\label{t:d2k0}
	Let $\Gamma$ be a subgroup of $O(2)$ with $k(\Gamma)=0$ and $|\Gamma| \geq 6$,
	and let $G$ be a $(2,0)$-tight multigraph.
	Then there exists a gain map $\phi : \vec{E} \rightarrow \Gamma$ such that $(G, \phi)$ is $\Gamma$-symmetrically rigid in $\mathbb{R}^2$.
\end{theorem}

\begin{proof}
	By \Cref{l:d2k},
	$\Gamma$ contains both rotations and reflections.
	By \Cref{t:ext2},
	$G$ can be formed from $K_1^2$ by a sequence of $0$-extensions, loop-0-extensions, $1$-extensions, loop-1-extensions, $2$-extensions and loop-2-extensions.
	Choose any distinct non-identity elements $\gamma_1,\gamma_2 \in \Gamma$ so that the group $\langle \gamma_1,\gamma_2 \rangle$ is dihedral.
	Define $\phi': \vec{E(K_1^2)} \rightarrow \Gamma$ so that the loops of $K_1^2$ have gains $\gamma_1$ and $\gamma_2$.
	It is immediate that $(K_1^2,\phi')$ is $\Gamma$-symmetrically rigid in $\mathbb{R}^2$.
	Moreover, we note that this implies loop-0-extensions preserve $\Gamma$-symmetric rigidity.
	By applying \Cref{l:gainext,l:gain2ext} inductively we see that there exists a gain map $\phi: \vec{E} \rightarrow \Gamma$ so that $(G,\phi)$ is $\Gamma$-symmetrically rigid in $\mathbb{R}^2$.	
\end{proof}

\subsection{Groups with 4 elements}

\Cref{t:d2k0} unfortunately does not hold when $|\Gamma|=4$.
For example, take the multigraph with two vertices and four non-loop edges.
The only possible gain map produces the cover graph shown in \Cref{fig:flexK44}.
However, the created gain graph is $\Gamma$-symmetrically flexible;
see \cite{Dixon,Bottema,Wunderlich} for more details.
We now show that quadruples of parallel edges are the only block to extending \Cref{t:d2k0}.
We first begin with a result of a similar nature to \Cref{l:extnotriple}.
\begin{figure}[ht]
    \centering
    \begin{tikzpicture}[gridl/.style={dotted,black!50!white}]
		\node[vertex] (1) at (1,1.5) {};
		\node[vertex] (2) at (-1,-1.5) {};
		\node[vertex] (3) at (1,-1.5) {};
		\node[vertex] (4) at (2,0.5) {};
		\node[vertex] (5) at (2,-0.5) {};
		\node[vertex] (6) at (-2,-0.5) {};
		\node[vertex] (7) at (-1,1.5) {};
		\node[vertex] (8) at (-2,0.5) {};

		\draw[gridl] (2)edge(3) (7)edge(1) (6)edge(5) (8)edge(4);
		\draw[gridl] (6)edge(8) (5)edge(4) (2)edge(7) (3)edge(1); 
		\draw[edge] (5)edge(3) (2)edge(6);
		\draw[edge] (2)edge(4)  (5)edge(7);
		\draw[edge] (2)edge(5) (7)edge(4) (3)edge(6);
		\draw[edge] (7)edge(6) (3)edge(4);
		\draw[edge] (5)edge(1) (1)edge(6) (1)edge(4);
		\draw[edge] (1)edge(8) (2)edge(8) (3)edge(8) (7)edge(8);
    \end{tikzpicture}
    \caption{Flexible symmetric framework of $K_{4,4}$.}
    \label{fig:flexK44}
\end{figure}
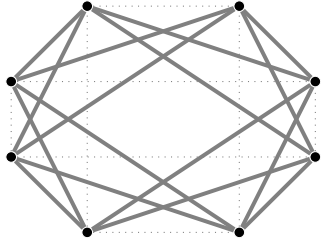

\begin{lemma}\label{l:extnofour}
	Let $G=(V,E)$ be a multigraph with no quadruple of parallel edges.
	Then $G$ is $(2,0)$-tight if and only if $G$ can be constructed from $K_1^{2}$ by a sequence of $0$-extensions, loop-0-extensions, $1$-extensions, loop-$1$-extensions, 2-extensions and loop-2-extensions that do not form a quadruple of parallel edges.
\end{lemma}

\begin{proof}
One direction follows immediately from \Cref{t:ext2}.
For the converse we first note that $G$ is $(2,0)$-tight if and only if every connected component is $(2,0)$-tight, so we may suppose $G=(V,E)$ is a connected $(2,0)$-tight multigraph distinct from $K_1^{2}$ with no quadruple of parallel edges. By applying the hand-shaking lemma we see that $G$ has a vertex of degree 2, 3 or 4. If there exists a degree 2 vertex, then the multigraph can be reduced to a smaller one with the required properties by a $0$-reduction. On the other hand, if $G$ has minimum degree 4 then the fact that $|E|=2|V|$ immediately implies that it is 4-regular. Fix a vertex $v$. Since $G \neq K_1^{2}$, the result of a 2-reduction at $v$ is 4-regular and hence $(2,0)$-tight (if a subgraph violates $(2,0)$-sparsity then the average degree in that subgraph is strictly greater than 4). Thus we are done unless we have added a 4th parallel edge between two vertices. This is only possible if the 2-reduction adds a 4th copy of an edge between two vertices. This implies that either there is a pair or a triple of parallel edges between two distinct neighbours $x,y$ of $v$ in $G$. In the former case it must be that there are two copies of the edge $vx$ and two copies of the edge $vy$, and so the reduction adding a loop on $x$ and a loop on $y$ is always possible. Suppose the latter case holds. Then $x$ has exactly two neighbours $y,v$ with a triple of parallel edges to $y$. Consider the 2-reduction at $x$ which adds a loop at $y$ and an edge between $yz$. The fact that $G$ is 4-regular implies that the multigraph $H$ resulting from this 2-reduction at $x$ is 4-regular and has no quadruple of parallel edges. 
    
So we may suppose the minimal degree of $G$ is 3.
If $G$ has a vertex of degree 3 incident to a loop then a loop 1-reduction is always possible at that vertex.
So suppose no degree 3 vertex of $G$ is incident to a loop. We now argue exactly as in \Cref{l:extnotriple}. It follows from the proof of \Cref{t:ext} (see \cite{feketeszego06}) that there is a 1-reduction at a vertex $v$ to a smaller $(2,0)$-tight multigraph. Suppose that every such 1-reduction at $v$ creates a quadruple of parallel edges. Note that two of the possible 1-reductions add only a new loop so cannot create a quadruple of parallel edges.

If $N(v)=\{x,y,z\}$ then without loss we may assume that $G$ contains three parallel edges between $x$ and $y$. We now consider the 1-reduction at $v$ that adds $xz$. This fails to result in a $(2,0)$-tight multigraph if and only if there is a $(2,0)$-tight subgraph $H$ of $G$ such that $x,z\in V(H)$ and $y,v\notin V(H)$. (To see that $y\notin V(H)$ simply note that if it was then $H+v$ would violate $(2,0)$-sparsity.)
However, if such a subgraph $H$ exists then we may add $y$ and its three parallel edges to $x$ to obtain a subgraph of $G$ that is not $(2,0)$-sparse.
Hence, the 1-reduction at $v$ adding $xz$ results in a $(2,0)$-tight multigraph and we are done unless there are already three parallel edges between $x$ and $z$. However, this implies that the subgraph of $G$ induced by $v$ and its neighbours has at least 9 edges and only 4 vertices, contradicting $(2,0)$-sparsity.

Therefore, $v$ has two distinct neighbours $x,y$ with a double edge from $v$ to $x$.
We now know two things from our assumptions:
(i) the vertex $x$ lies in a $(2,0)$-tight subgraph $H$ of $G$ (as otherwise we could perform the 1-reduction at $v$ that adds a loop at $x$ which does not form a triple of parallel edges), and
(ii) there exists a triple of parallel edges between $x,y$ (since the only remaining possible 1-reduction at $v$ to a smaller $(2,0)$-tight multigraph adds an edge between $x,y$, and this must form a quadruple of parallel edges).
However, the subgraph of $G$ formed from adding the vertex $y$ and the edges between the vertices $x,y$ to $H$ is not $(2,0)$-tight,
contradicting our original assumption. 
This completes the proof.
\end{proof}

\begin{lemma}\label{l:gain2extnoquad}
	Let $(G,\phi)$ be a $\Gamma$-gain graph where $\Gamma$ is a dihedral group,
	and let $G'=(V',E')$ be a multigraph formed from $G$ by either a $0$-extension, $1$-extension, loop-$1$-extension, $2$-extension or loop-$2$-extension.
	Suppose neither $G$ or $G'$ contain a quadruple of parallel edges.
	If $(G,\phi)$ is $\Gamma$-symmetrically rigid in $\mathbb{R}^2$ then there exists a gain map $\phi' :\vec{E'} \rightarrow \Gamma$ where $\phi'(e,v,w)=\phi(e,v,w)$ for all $(e,v,w) \in \vec{E} \cap \vec{E}'$ and $(G',\phi')$ is $\Gamma$-symmetrically rigid in $\mathbb{R}^2$.
\end{lemma}

\begin{proof}
    This follows from \Cref{l:gainext,l:gain2ext} by noting that the only extension moves that require $|\Gamma| \geq 5$ are those that create a quadruple of parallel edges.
\end{proof}

\begin{theorem}\label{thm:2dsmall}
	Let $\Gamma$ be a subgroup of $O(2)$ with $k(\Gamma)=0$ and $|\Gamma| =4$,
	and let $G$ be a $(2,0)$-tight multigraph with no quadruple of parallel edges.
	Then there exists a gain map $\phi : \vec{E} \rightarrow \Gamma$ such that $(G, \phi)$ is $\Gamma$-symmetrically rigid in $\mathbb{R}^2$.
\end{theorem}

\begin{proof}
	By \Cref{l:extnofour},
	$G$ can be formed from $K_1^2$ by a sequence of $0$-extensions, loop-0-extensions, $1$-extensions, loop-$1$-extensions, 2-extensions and loop-2-extensions that do not form a quadruple of parallel edges.
	Choose any distinct non-identity elements $\gamma_1,\gamma_2 \in \Gamma$ so that the group $\langle \gamma_1,\gamma_2 \rangle$ is dihedral.
	$\gamma_1,\gamma_2 \in \Gamma$ so that $\langle \gamma_1,\gamma_2 \rangle = \Gamma$.
	Define $\phi': \vec{E(K_1^2)} \rightarrow \Gamma$ so that the loops of $K_1^2$ have gains $\gamma_1$ and $\gamma_2$.
	It is immediate that $(K_1^2,\phi')$ is $\Gamma$-symmetrically rigid in $\mathbb{R}^2$.
	Moreover, we note that this implies that loop-0-extensions preserve $\Gamma$-symmetric rigidity.
	By applying \Cref{l:gain2extnoquad} inductively we see that there exists a gain map $\phi: \vec{E} \rightarrow \Gamma$ so that $(G,\phi)$ is $\Gamma$-symmetrically rigid in $\mathbb{R}^2$.
\end{proof}

We remark that a special case of \Cref{t:d2k0,thm:2dsmall} partially resolves the even order dihedral symmetry problem. Motivated by potential future strengthening of these results to a complete solution we  consider, in \Cref{sec:prob}, how many gain assignments result in symmetrically rigid frameworks.

\section{Higher dimensional symmetry groups}
\label{sec:highd}

In the previous sections we have focused on symmetry groups in the plane.
We now instead focus on symmetry groups in $d$-space.
We begin the section by investigating symmetry groups of translations,
where we recap a result of \cite{ww88}.
This illustrates some of the techniques we use later on when we move on to point groups that are dense in $O(d)$.

\subsection{Periodic frameworks}

Periodic structures arise naturally in material science and crystallography, where the rigidity of polytope networks is of particular interest \cite{Born, Wegner}. During the last 15 years, these structures have been studied from a theoretical approach, since periodicity admits a block diagonalization of the rigidity matrix over the integers. These diagonal blocks give rise to periodic flexes, which can be considered in the fixed \cite{Ross2014} or the flexible lattice case \cite{BS10,MT13}. More results from an operator-theoretic perspective were initiated in \cite{OwenPower11}, which gave applications on factor periodic flexes and the rigid unit mode (RUM) spectrum were initiated. Further studies on the RUM spectrum are discussed in \cite{KKM21}.

We first give a proof of a reformulation of \cite[Theorem 13]{ww88}, as a warm up for what follows. 

\begin{theorem}\label{t:periodic}
	Let $G$ be a $(d,d)$-tight multigraph and $\Gamma$ be a group of translations with $d$ linearly independent generators.
	Then there exists a gain map $\phi : \vec{E} \rightarrow \Gamma$ such that $(G, \phi)$ is $\Gamma$-symmetrically rigid in $\mathbb{R}^d$.
\end{theorem}

	It is worth noting that this is not the same as the main theorem given in \cite{borceastreinu11} since we are not allowing the lattice of periodicity to deform.

\begin{proof}
	Since $\Gamma$ is a group of translations,
	we consider $\Gamma$ to be a subgroup of $\mathbb{R}^d$ (with respect to addition).
	We also define $[x]_j$ to be the $j$-th coordinate of any vector $x \in \mathbb{R}^d$.
	
	As $G$ is $(d,d)$-tight then, by Nash-Williams theorem \cite{nashwilliams61}, there exist $d$ edge-disjoint spanning trees $T_1,\ldots,T_d$.
	We shall assign each tree some arbitrary orientation so that they are directed. Let $\gamma_1,\ldots,\gamma_d$ denote linearly independent generators of $\Gamma$. We define $\phi : \vec{E} \rightarrow \Gamma$ to be the gain map where,
	given $(e,v,w) \in T_i$,
	we have $\phi(e,v,w) = \gamma_i$.
	We now define the $|E| \times d|V|$ matrix $M$ with entries,
	where
	\begin{align*}
		M_{e,(v,j)}
		:= 
		\begin{cases}
			-[\gamma_i]_j &\text{if } (e,v,w) \in T_i, \\
			[\gamma_i]_j &\text{if } (e,w,v) \in T_i, \\
			0 &\text{otherwise}.
		\end{cases}
	\end{align*}
	As $\gamma_1,\ldots,\gamma_d$ are linearly independent,
	we have that $\ker M$ is exactly the vectors $(x)_{v\in V, ~ 1 \leq j\leq d}$ for $x \in \mathbb{R}^d$.
	
	Let $\mathbf{p} :\mathbf{V} \rightarrow \mathbb{R}^d$ be the placement of $\mathbf{G}$ with $\mathbf{p}(v,0) = 0$ for all $v \in V$.
	We now note that $M = R(G,\phi,p)$,
	hence $(\mathbf{G},\mathbf{p})$ is $\Gamma$-symmetrically rigid as required.
\end{proof}

\subsection{Periodic frameworks with additional symmetry}

We next extend \Cref{t:periodic} to frameworks with additional symmetry. We do this in the next two theorems for all groups containing translations and certain additional linear isometries. In both cases we use a combinatorial decomposition of $(d,0)$-tight graphs into $d$ $(1,0)$-tight spanning subgraphs (see \cite[Corollary 3]{ww88}). In the first theorem we need the additional assumption that these subgraphs are connected; we conjecture that this hypothesis can be removed.

\begin{theorem}\label{t:trans+lin}
    Let $G=(V,E)$ be a $(d,0)$-tight multigraph with a spanning $(d,d)$-tight subgraph $H=(V,F)$,
    and let $\Gamma$ be a $d$-dimensional symmetry group generated by a group $\Gamma_t$ of translations with $d$ linearly independent generators and a point group $\Gamma_\ell \leq O(d)$ where $\bigcap_{\gamma \in \Gamma_\ell}\ker (I - \gamma) = \{0\}$.
	Then there exists a gain map $\phi : \vec{E} \rightarrow \Gamma$ such that $(G, \phi)$ is $\Gamma$-symmetrically rigid in $\mathbb{R}^d$.
\end{theorem}

While the condition that $\bigcap_{\gamma \in \Gamma_\ell}\ker (I - \gamma) = \{0\}$ is needed for technical reasons in the proof that follows, it is easy to see that many groups satisfy this condition. For example, if $d=3$, it is sufficient that $\Gamma_\ell$ contains either a rotation with no fixed axis, or two rotations with disjoint fixed axes.

\begin{proof}
    Let $p:V \rightarrow \mathbb{R}^d$ be the map where $p(v) = 0$ for all $v \in V$.
    By using the methods from \Cref{t:periodic},
    we see that there exists a gain map $\phi_t : \vec{F} \rightarrow \Gamma_t$ such that $(H, \phi_t,p)$ is $\Gamma_t$-symmetrically rigid in $\mathbb{R}^d$.
    Label the remaining edges $e_1, \ldots, e_d \in E \setminus F$,
    and let $v_i,w_i$ be the end-points of $e_i$.
    
    Fix $d$ linearly independent translations $\gamma_1,\ldots, \gamma_d \in \Gamma_t$,
    and let $\{x_1,\ldots, x_d\} \subset \mathbb{R}^d$ be the basis where $\gamma_i(x) = x + x_i$ for each $i \in \{1,\ldots,d\}$.
    For every $i \in \{1,\ldots,d\}$ and $\gamma \in \Gamma_\ell$,
    define the linear space $V_{i,\gamma} := \{ x \in \mathbb{R}^d : (I - \gamma^{-1})(x_i) \cdot x  = 0 \}$.
    We note two things;
    if $x_i \notin \ker (I - \gamma^{-1})$,
    then $V_{i,\gamma}$ has dimension $d-1$,
    and $\bigcap_{i = 1}^d V_{i,\gamma} = \ker(I - \gamma^{-1})$.
    Hence, there exist $n_1,\ldots,n_d \in \{1,\ldots,d\}$ and $\mu_1,\ldots,\mu_d \in \Gamma_\ell$ where $\bigcap_{i=1}^d V_{n_i, \gamma_i} = \{0\}$.
    
    Define $\phi:\vec{E} \rightarrow \Gamma$, where $\phi(e,v,w) = \phi_t(e,v,w)$ for each $e \in F$,
    and $\phi(e_i,v_i,w_i) = \gamma_{n_i} \circ \mu_i$ for each $i \in \{1,\ldots,d\}$.
    Now choose a $\Gamma$-symmetric flex $u : V \rightarrow \mathbb{R}^d$ of $(G,\phi,p)$.
    Since $(H, \phi_t,p)$ is $\Gamma_t$-symmetrically rigid in $\mathbb{R}^d$,
    there exists $z \in \mathbb{R}^d$ so that $u(v) = z$ for all $v \in V$.
    For each edge $(e_i,v_i,w_i)$,
    we note that
    \begin{align*}
        0 = \Big(p(v_i) - \gamma_{n_i} \circ \mu_i(p(w_i))\Big) \cdot u(v_i) + \Big(p(w_i) - \mu_i^{-1} \circ \gamma_{n_i}^{-1} (p(v_i))\Big) \cdot u(w_i) = -(I - \mu_i^{-1})(x_{n_i}) \cdot z, 
    \end{align*}
    and so $z \in \bigcap_{i=1}^d V_{n_i, \gamma_i} = \{0\}$.
    Hence, $u$ is trivial and $(G,\phi,p)$ is $\Gamma$-symmetrically rigid in $\mathbb{R}^d$.
\end{proof}

We can improve \Cref{t:trans+lin} when $\Gamma_\ell$ contains the reflection through the origin.

\begin{theorem}\label{t:trans+inv}
    Let $G=(V,E)$ be a $(d,0)$-tight multigraph,
    and let $\Gamma$ be a $d$-dimensional symmetry group containing a subgroup $\Gamma_t$ of translations with $d$ linearly independent generators and the map $-I : x \mapsto -x$.
	Then there exists a gain map $\phi : \vec{E} \rightarrow \Gamma$ such that $(G, \phi)$ is $\Gamma$-symmetrically rigid in $\mathbb{R}^d$.
\end{theorem}

\begin{proof}
    As $G$ is $(d,0)$-tight,
    we can partition $G$ into $d$ spanning $(1,0)$-tight graphs $G_1,\ldots G_d$.
    For each $G_i = (V, E_i)$,
    let $G_{i,1}  = (V_{i,1},E_{i,1}),\ldots , G_{i,n_i}  = (V_{i,d},E_{i,d})$ be the connected components of $G_i$,
    let $T_{i,j} = (V_{i,j},F_{i,j})$ be the spanning tree of $G_{i,j}$ and let $e_{i,j}$ be the unique edge in $E_{i,j} \setminus F_{i,j}$.
    We shall assign each edge some arbitrary orientation, and we label the source and sink of each edge $e_{i,j}$ as $v_{i,j}$ and $w_{i,j}$ respectively.
    Now define $\phi:\vec{E} \rightarrow \mathbb{R}^d$ be the $\Gamma$-gain map, where $\phi(e,v,w) = \gamma_i$ for every edge $e \in F_{i,j}$,
    and $\phi(e_{i,j},v_{i,j},w_{i,j}) = \gamma_i \circ (-I)$.
    
    Fix $p:V \rightarrow \mathbb{R}^d$ where $p(v) = 0$ for all $v \in V$,
    and choose a $\Gamma$-symmetric flex $u : V \rightarrow \mathbb{R}^d$ of $(G,\phi,p)$.
    Fix $i,j$ and choose any adjacent vertices $v,w \in V_{i,j}$ in $T_{i,j}$.
    Given the edge $e \in F_{i,j}$ has source $v$ and sink $w$, we have that
    \begin{align*}
        x_i \cdot (u(w) - u(v)) = (p(v) - \gamma_i p(w)) \cdot u(v) + (p(w) - \gamma_i^{-1} p(v)) \cdot u(w)  = 0,
    \end{align*}
    hence, $x_i \cdot u(v) = x_i \cdot u(w)$.
    It now follows from transitivity that $x_i \cdot u(v) = x_i \cdot u(w)$ for any pair of vertices $v,w \in V_{i,j}$.
    From observing the edge condition for $(e_{i,j},v_{i,j},w_{i,j})$,
    we see that
    \begin{eqnarray*}
        0 &=& (p(v_{i,j}) - (\gamma_i \circ -I) p(w_{i,j})) \cdot u(v_{i,j}) + (p(w_{i,j}) - (\gamma_i \circ -I)^{-1} p(v_{i,j})) \cdot u(w_{i,j}) \\
        &=& - x_i \cdot u(v_{i,j})- x_i \cdot u(w_{i,j}),
    \end{eqnarray*}
    hence, $x_i \cdot u(v_{i,j}) = - x_i \cdot u(w_{i,j})$.
    As $v_{i,j}, w_{i,j} \in V_{i,j}$,
    it follows that $x_i \cdot u(v_{i,j}) = x_i \cdot u(w_{i,j}) = 0$,
    hence, $x_i \cdot u(v) = 0$ for all $v \in V_{i,j}$.
    By applying this method for each $j \in \{1,\ldots, n_i\}$,
    we have that $x_i \cdot u(v) = 0$ for all $v \in V$.
    For each $v \in V$ we have $x_i \cdot u(v) = 0$ for each $i \in \{1,\ldots,d\}$,
    and thus $u(v)=0$.
    Hence, $u$ is trivial and $(G,\phi,p)$ is $\Gamma$-symmetrically rigid in $\mathbb{R}^d$.
\end{proof}

\subsection{Infinite point groups}

We next consider infinite point groups with a suitable density property.
For the following we remember that for higher dimensions,
a linear isometry $\gamma \in O(d)$ is a \emph{rotation} if $\det \gamma = 1$,
and we denote by $SO(d)$ the set of all $d$-dimensional rotations.
A \emph{reflection} is any linear isometry $\sigma \in O(d)$ with $\det \sigma =-1$ that has a linear hyperplane of fixed points.
An important property of reflections is that they are involutory (i.e.~$\sigma^{-1} =\sigma$) and $\gamma^{-1} \sigma \gamma$ is a reflection for any rotation $\gamma$.

\begin{lemma}\label{lem:approxisoms}
    Let $\Gamma \leq O(d)$ be a dense subgroup.
    Choose an orthonormal basis $f_1,\ldots,f_d$ of $\mathbb{R}^d$.
    Then for every $\varepsilon > 0$,
    there exist rotations $\gamma_1,\ldots,\gamma_{d-1} \in \Gamma$ and isometries $\sigma_1,\ldots,\sigma_{d-1}\in \Gamma$ so that for all $k \in \{1,\ldots,d-1\}$ we have
    \begin{align*}
	   \left\|-f_k - \frac{f_d-\gamma_k f_d}{\|f_d-\gamma_k f_d\|} \right\| &<\varepsilon, & \left\|f_k - \frac{f_d-\gamma_k^{-1} f_d}{\|f_d-\gamma_k f_d\|} \right\|  &<\varepsilon, \\
	   \left\|-f_k - \frac{f_d-\sigma_k f_d}{\|f_d-\sigma_k f_d\|} \right\| &<\varepsilon, & \left\|-f_k - \frac{f_d-\sigma_k^{-1} f_d}{\|f_d-\sigma_k f_d\|} \right\|  &<\varepsilon, &
	   \left\|-f_k - \frac{2f_d-\sigma_k f_d -\sigma_k^{-1} f_d}{2\|f_d-\sigma_k f_d\|} \right\|  &<\varepsilon,
	\intertext{and there exists an isometry $\sigma_d \in \Gamma$ so that}
	   \left\|f_d - \frac{f_d-\sigma_d f_d}{\|f_d-\sigma_d f_d\|} \right\| &<\varepsilon, & \left\|f_d - \frac{f_d-\sigma_d^{-1} f_d}{\|f_d-\sigma_d f_d\|} \right\|  &<\varepsilon, &
	   \left\|f_d - \frac{2f_d-\sigma_d f_d -\sigma_d^{-1} f_d}{2\|f_d-\sigma_d f_d\|} \right\|  &<\varepsilon.
	\end{align*}
\end{lemma}

\begin{proof}
    Let $\sigma$ be the reflection with $\sigma f_d=-f_d$.
    For each $k \in \{1,\ldots,d-1\}$ and $0 \leq \theta <2\pi$,
    define $R_k(\theta) \in SO(d)$ to be the rotation in the $f_kf_d$-plane where $R_k(\theta)f_d = (\cos \theta) f_d + (\sin \theta)f_k$.
    Similarly,
    define for each $k \in \{1,\ldots,d-1\}$ the reflection $S_k(\theta) := R_k(\pi-\theta/2)^{-1} \sigma R_k(\pi-\theta/2)$.
    Then for suitably small $\theta_k$ we have
    \begin{align*}
	   \left\|-f_k - \frac{f_d-R_k(\theta_k) f_d}{\|f_d-R_k(\theta_k) f_d\|} \right\| &<\varepsilon, & \left\|f_k - \frac{f_d-R_k(\theta_k)^{-1} f_d}{\|f_d-R_k(\theta_k) f_d\|} \right\|  &<\varepsilon,\\
	   \left\|-f_k - \frac{f_d-S_k(\theta_k) f_d}{\|f_d-S_k(\theta_k) f_d\|} \right\| &<\varepsilon, & \left\|-f_k - \frac{f_d-S_k(\theta_k)^{-1} f_d}{\|f_d-S_k(\theta_k) f_d\|} \right\|  &<\varepsilon,\\
	   &&\left\|-f_k - \frac{2f_d-S_k(\theta_k) f_d -S_k(\theta_k)^{-1} f_d}{2\|f_d-S_k(\theta_k) f_d\|} \right\|  &<\varepsilon.
	\end{align*}
    Since $\Gamma$ is dense in $O(d)$, we can now choose for each $k \in \{1,\ldots,d-1\}$ the rotation $\gamma_k \in \Gamma$ suitably close to $R_k(\theta_k)$ so as to satisfy the required properties.
    Similarly, we can choose for $k \in \{1,\ldots,d-1\}$ the isometry $\sigma_k \in \Gamma$ suitably close to $S_k(\theta_k)$ so as to satisfy the required properties also.
    Finally,
    since
    \begin{align*}
	   \left\|f_d - \frac{f_d-\sigma f_d}{\|f_d-\sigma f_d\|} \right\| = \left\|f_d - \frac{f_d-\sigma^{-1} f_d}{\|f_d-\sigma f_d\|} \right\|  =
	   \left\|f_d - \frac{2f_d-\sigma f_d -\sigma f_d}{2\|f_d-\sigma f_d\|} \right\| =0,
	\end{align*}
    we may choose $\sigma_d$ to be suitably close to $\sigma$ to satisfy the required properties.
\end{proof}

\begin{theorem}\label{thm:densegroup}
    Let $G$ be a $(d,0)$-tight multigraph and let $\Gamma \leq O(d)$ be a dense subgroup.
	Then there exists a gain map $\phi : \vec{E} \rightarrow \Gamma$ such that $(G, \phi)$ is $\Gamma$-symmetrically rigid in $\mathbb{R}^d$.
\end{theorem}

\begin{proof}
	As $G$ is $(d,0)$-tight, it can be decomposed into $d$ edge-disjoint spanning $(1,0)$-tight subgraphs $T_1,\ldots,T_d$ (see \cite[Corollary 3]{ww88}).
	Define $n_1,\ldots,n_d$ to be the number of connected components of $T_1,\ldots,T_d$ respectively,
	and define $T_k^{1}, \ldots, T_k^{n_k}$ to be the connected components of $T_k$.
	For each $T_k^j$,
	fix an edge $e_k^j$ in the unique cycle\footnote{A multigraph is $(1,0)$-tight if and only if every connected component contains exactly one cycle.} in $T_k^j$.
	Label the vertices of $G$ as $v_1,\ldots, v_n$.
	
	Fix a unital basis $f_1,\ldots,f_d$ where $\sigma f_d= -f_d$, and choose any $\varepsilon >0$.
	Let $\gamma_1,\ldots,\gamma_{d-1}$ and $\sigma_1,\ldots,\sigma_{d-1}$ be the rotations and reflections described in \Cref{lem:approxisoms},
 define $\gamma_d$ to be the identity map and define $\sigma_d$ to be sufficiently close to $\sigma$, as demanded in \Cref{lem:approxisoms}.
	We define $\phi_\varepsilon : \vec{E} \rightarrow \Gamma$ to be the gain map where for any edge $(e,v_i,v_j) \in T_k$ with $i\leq j$,
	we have 
	\begin{align*}
	    \phi_\varepsilon(e,v_i,v_j) =
	    \begin{cases}
	        \gamma_k &\text{if } k \in \{1,\ldots,d\} \text{ and } e \notin \{ e_k^1,\ldots, e_k^{n_k}\},\\
	        \sigma_k &\text{if } k \in \{1,\ldots,d\} \text{ and } e \in \{ e_k^1,\ldots, e_k^{n_k}\}.
	    \end{cases}
	\end{align*}
	We note that every loop has gain $\sigma_k$ for some $k \in \{1,\ldots,d\}$.
	
	Define the $|E| \times d|V|$ matrix $M$ with entries,
	where for every vertex $v \in V$ and for every edge $e \in T_k$ with end points $v_i,v_j$ ($i\leq j$),
	\begin{align*}
		M_{e,v}
		:= 
		\begin{cases}
			-f_k &\text{if } v=v_i \text{ and } e \notin \{ e_k^1,\ldots, e_k^{n_k}\}, \\
			f_k  &\text{if } v=v_j \text{ and } e \notin \{ e_k^1,\ldots, e_k^{n_k}\}, \\
			-f_k &\text{if } v\in \{ v_i,v_j\} \text{ and } e \in \{ e_k^1,\ldots, e_k^{n_k}\}, \\
			(0,0,0)  &\text{otherwise}.
		\end{cases}
	\end{align*}
	(This is a slight abuse of notation, as the entries here are $d$-coordinate vectors.)
	By construction it is evident that $\rank M = d|V|$.
	
	Choose $0<\delta<\varepsilon$ and let $p :V \rightarrow \mathbb{R}^d$ be the map where $p(v_i)= (1+ \delta i) f_d$ for all $v_i \in V$.
	Choose a non-looped edge $(e,v_i,v_j)$ ($i<j$) with $\phi_\varepsilon(e,v_i,v_j) = \gamma_k$ for $k \neq d$.
	Then
	\begin{eqnarray*}
	    \bigg\| p(v_i) - \gamma_k p(v_j) + \| f_d- \gamma_k f_d\| f_k \bigg\| &=&
	    \bigg\| (1+ \delta i) f_d-\gamma_k((1+ \delta j) f_d)  + \| f_d- \gamma_k f_d\| f_k \bigg\| \\
	    &=& \bigg\| (f_d-\gamma_k f_d)  + \| f_d- \gamma_k f_d\| f_k + \delta (i f_d- j \gamma_k f_d ) \bigg\|\\
	    &\leq& \| f_d- \gamma_k f_d\|  \left\|-f_k - \frac{f_d-\gamma_k f_d}{\|f_d-\gamma_k f_d\|} \right\| + \delta \left\| i f_d- j \gamma_k f_d \right\| \\
	    &<& \varepsilon \| f_d- \gamma_k f_d\| + \delta \left\| i f_d- j \gamma_k f_d \right\|,
	\end{eqnarray*}
	and hence, we may choose $\delta >0$ sufficiently small so that
	\begin{align*}
	    \left\|\frac{p(v_i) - \gamma_k p(v_j)}{\| f_d- \gamma_k f_d\|} - (-f_k) \right\| < \varepsilon.
	\end{align*}
	By a similar method,
	we can choose a sufficiently small $\delta$ to ensure 
	\begin{align*}
	    \left\|\frac{p(v_j) - \gamma_k^{-1} p(v_i)}{\| f_d- \gamma_k f_d\|} - f_k \right\| < \varepsilon
	\end{align*}
	holds.
	Likewise,
	for any edge $(e,v_i,v_j)$ ($i \leq j$) with $\phi_\varepsilon(e,v_i,v_j) = \sigma_k$ and $k \neq d$,
	we can choose $\delta$ sufficiently small so that
	\begin{align*}
	    \left\|\frac{p(v_i) - \sigma_k p(v_j)}{\| f_d- \sigma_k f_d\|} - (-f_k) \right\| < \varepsilon, \qquad
	    \left\|\frac{p(v_j) - \sigma_k^{-1} p(v_i)}{\| f_d- \sigma_k f_d\|} - (-f_k) \right\| < \varepsilon, \\
	    \left\|\frac{2p(v_i) - \sigma_k p(v_i) - \sigma_k^{-1} p(v_i)}{2\| f_d- \sigma_k f_d\|} - (-f_k) \right\| < \varepsilon.
	\end{align*}
	Finally,
	for any edge $(e,v_i,v_j)$ ($i \leq j$) with $\phi_\varepsilon(e,v_i,v_j) = \sigma_d$,
	we can choose $\delta$ sufficiently small so that
	\begin{align*}
	    \left\|\frac{p(v_i) - \sigma_d p(v_j)}{\| f_d- \sigma_d f_d\|} - f_d \right\| < \varepsilon, \qquad
	    \left\|\frac{p(v_j) - \sigma_d^{-1} p(v_i)}{\| f_d- \sigma_k f_d\|} - f_d \right\| < \varepsilon, \\
	    \left\|\frac{2p(v_i) - \sigma_d p(v_i) - \sigma_d^{-1} p(v_i)}{2\| f_d- \sigma_d f_d\|} - f_d \right\| < \varepsilon.
	\end{align*}
	Define $\tilde{R}(G,\phi_\varepsilon,p)$ to be the matrix formed from $R(G,\phi_\varepsilon,p)$ by multiplying each row corresponding to an edge $(e,v_i,v_j)$ by $\alpha$,
	where (i) $\alpha =\|f_d - \lambda f_d\|$ if it is a non-loop edge with gain $\lambda \notin \{ I, \sigma_d,\sigma_d^{-1}\}$,
	(ii) $\alpha =2\|f_d - \lambda f_d\|$ if it is a loop with gain $\lambda \neq \sigma_d$,
	(iii) $\alpha =-\|f_d - \sigma_d f_d\|$ if it is a non-loop edge with gain $\sigma_d$ or $\sigma_d^{-1}$,
	(iv) $\alpha =-2\|f_d - \sigma_d f_d\|$ if it is a loop with gain $\sigma_d$,
	and 
	(v) $\alpha = \delta|i-j|$ if it is an edge $(e,v_i,v_j)$ with trivial gain.
	By our choice of entries,
	every entry of $\tilde{R}(G,\phi_\varepsilon,p)$ lies within $\varepsilon$ of the corresponding entry in the matrix $M$.
	Since $M$ is non-singular, we must have for sufficiently small $\varepsilon$ that
	\begin{align*}
	    \rank R(G,\phi_\varepsilon,p) = \rank \tilde{R}(G,\phi_\varepsilon,p) = \rank M = d|V|.
	\end{align*}
	Now fix $\phi = \phi_\varepsilon$ for some $\varepsilon >0$ where the above holds.
	Then $(G,\phi,p)$ is $\Gamma$-symmetrically rigid as required.
\end{proof}

\begin{example}
    Define the rotations
    \begin{align*}
        \rho_1=
        \frac{1}{5}
        \begin{pmatrix} 3 & 4 & 0\\-4 & 3 & 0\\ 0 & 0 & 5 \end{pmatrix},
        \qquad
        \rho_2=
        \frac{1}{5}
        \begin{pmatrix} 5 & 0 & 0 \\0 & 3 & 4\\0 & -4 & 3 \end{pmatrix}.
    \end{align*}
    The element $\rho_1$ defines an irrational rotation in the $xy$-plane,
    while the element $\rho_2$ defines an irrational rotation in the $yz$-plane.
    The group $\langle \rho_1,\rho_2 \rangle$ is isomorphic to $F_2$ (see \cite{Tao04}) and dense in $SO(3)$;
    the latter property can be seen from observing that the closure of $\langle \rho_1 \rangle$ contains every rotation in the $xy$-plane and the closure of $\langle \rho_2 \rangle$ contains every rotation in the $yz$-plane.
    We can now apply \Cref{thm:densegroup} to the group $\Gamma := \langle \rho_1,\rho_2,\sigma \rangle$ for any choice of reflection $\sigma$.
    Note that any finitely generated group $F$ that is dense in $SO(3)$ contains a dense free subgroup of rank~$2$~\cite{BreGe03}. 
\end{example}

\subsection{Large finite point groups}

Let $A$ be a subset of a metric space $(M,d)$.
For a given $\varepsilon >0$, we say that $A$ is an \emph{$\varepsilon$-dense} subset of $M$ if for every $x \in M$ there exists $a \in A$ so that $d(a,x) <\varepsilon$.
By observing the proofs of \Cref{lem:approxisoms} and \Cref{thm:densegroup},
we can easily see that the following generalisation is also true.

\begin{theorem}\label{lem:epsilondensegroup}
    Let $G$ be a $(d,0)$-tight multigraph.
	Then there exists $\varepsilon>0$ so that the following holds for any $\varepsilon$-dense subgroup $\Gamma \leq O(d)$;
	there exists a gain map $\phi : \vec{E} \rightarrow \Gamma$ such that $(G, \phi)$ is $\Gamma$-symmetrically rigid in $\mathbb{R}^d$.
\end{theorem}

\begin{proof}
Define the matrix $M$ as in \Cref{thm:densegroup}.
Since the matrix $M$ is non-singular, there exists an $\varepsilon>0$, such that for every matrix $R$ that satisfies $\|R-M\|\leq \varepsilon$, we have $\operatorname{rank}R=\operatorname{rank}M$.
Let $\Gamma$ be an $\varepsilon$-dense subgroup in $O(d)$. Then there exist rotations $\gamma_1,\ldots,\gamma_{d-1} \in \Gamma$ and isometries $\sigma_1,\ldots,\sigma_{d-1}, \sigma_d\in \Gamma$, satisfying the properties of \Cref{lem:approxisoms}. Hence, by repeating the proof of  \Cref{thm:densegroup}, the result follows. 
\end{proof}




Given a $(d,0)$-tight multigraph, it would be interesting to quantify how large a point group $\Gamma$, where $k(\Gamma)=0$, needs to be for there to exist a rigid gain map.

\section[Probabilities for rigid gain assignments]{Probabilities for rigid gain assignments}
\label{sec:prob}

Let $G$ be a multigraph and let $\Gamma \leq O(d)$ be a finite group.
We say that a $\Gamma$-gain map $\phi$ of $G$ is a \emph{rigid gain assignment} of $G$ if $(G,\phi)$ is $\Gamma$-symmetrically rigid.
We define $\mathbb{P}(G,\Gamma) \in [0,1)$ to be the probability that a random gain map $\phi:\vec{E} \rightarrow \Gamma$ is a rigid gain assignments.
If $\Gamma$ is the identity group then $\mathbb{P}(G,\Gamma) \in \{0,1\}$ since there is exactly one gain map. Suppose $\Gamma$ is non-trivial. Note that for any simple graph $G$ we must have $\mathbb{P}(G,\Gamma)<1$ as the gain map $\phi$ that assigns every edge trivial gain forces the gain graph $(G,\phi)$ to be $\Gamma$-symmetrically flexible.

For a group $\Gamma$ containing just a few elements, the probability of choosing gains which result in a $\Gamma$-symmetrically rigid framework can be experimentally checked using a computer algebra system. Due to the number of possible gain maps, this method, however, gets infeasible even for moderate numbers of vertices.
While we show later that the relative number of $\Gamma$-symmetrically rigid choices of gains can be arbitrarily low, we see that for a small number of vertices, and for two basic groups $\Gamma$, we get probabilities not lower than $0.96$ for rigidity of a random choice.

\Cref{fig:dim2rot} shows some multigraphs with different probabilities for gains to be $\Gamma$-symmetrically rigid with respect to the group of 90 degree rotations in the plane. We were able to compute the probability for all possible multigraphs with at most five vertices. To get suitable multigraphs we need to check all multigraphs with the respective number of vertices for $(2,1)$-tightness and then check all gains for $\Gamma$-symmetric rigidity. On 5 vertices there are $3765$ $(2,1)$-tight multigraphs and up to $4^9$ possible gain maps. However, already on six vertices the numbers seem out of computational reach: there are $281384$ non-isomorphic $(2,1)$-tight multigraphs, and for each of these multigraphs there are many possible gain maps, with exactly $4^{11}$ (over 4 million) possibilities in the worst case\footnote{The total number of gain maps depends on the number of multiple edges, with $4^{11}$ achieved for any simple graph.}.

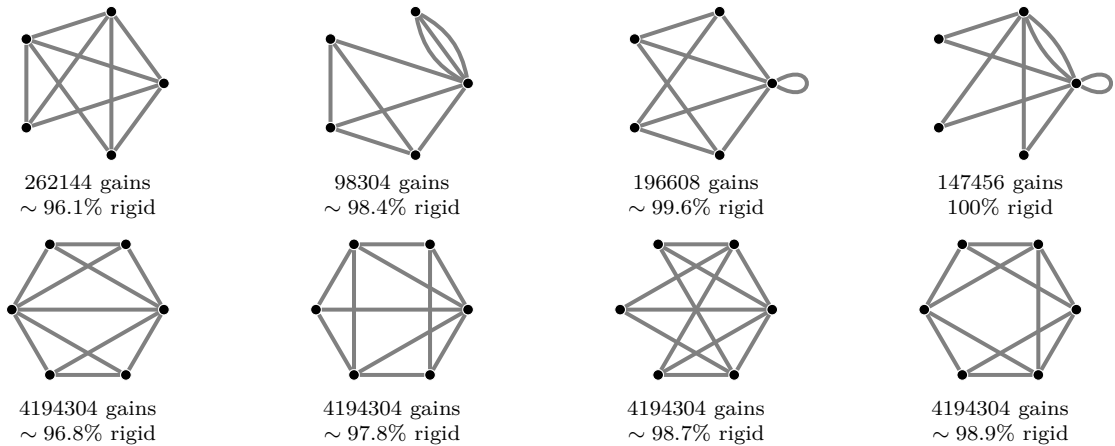
\begin{figure}[ht]
    \centering
    \begin{tikzpicture}[every loop/.style={min distance=6mm,looseness=5}]
        \begin{scope}
            \node[vertex] (1) at (0:1) {};
            \node[vertex] (2) at (72:1) {};
            \node[vertex] (3) at (144:1) {};
            \node[vertex] (4) at (216:1) {};
            \node[vertex] (5) at (288:1) {};
            \draw[edge] (1)to(2);
            \draw[edge] (1)to(3);
            \draw[edge] (1)to(4);
            \draw[edge] (1)to(5);
            \draw[edge] (2)to(3);
            \draw[edge] (2)to(4);
            \draw[edge] (2)to(5);
            \draw[edge] (3)to(4);
            \draw[edge] (3)to(5);
            \node[labelsty,align=center] at (0,-1.5) {262144 gains\\$\sim96.1\%$ rigid};
        \end{scope}
        \begin{scope}[xshift=4cm]
            \node[vertex] (1) at (0:1) {};
            \node[vertex] (2) at (72:1) {};
            \node[vertex] (3) at (144:1) {};
            \node[vertex] (4) at (216:1) {};
            \node[vertex] (5) at (288:1) {};
            \draw[edge] (1)to[bend right=25] (2);
            \draw[edge] (1)to(2);
            \draw[edge] (1)to[bend right=-25] (2);
            \draw[edge] (1)to(3);
            \draw[edge] (1)to(4);
            \draw[edge] (1)to(5);
            \draw[edge] (3)to(4);
            \draw[edge] (3)to(5);
            \draw[edge] (4)to(5);
            \node[labelsty,align=center] at (0,-1.5) {98304 gains\\$\sim98.4\%$ rigid};
        \end{scope}
        \begin{scope}[xshift=8cm]
            \node[vertex] (1) at (0:1) {};
            \node[vertex] (2) at (72:1) {};
            \node[vertex] (3) at (144:1) {};
            \node[vertex] (4) at (216:1) {};
            \node[vertex] (5) at (288:1) {};
            \draw[edge] (1)to(2);
            \draw[edge] (1)to(3);
            \draw[edge] (1)to(4);
            \draw[edge] (1)to(5);
            \draw[edge] (2)to(3);
            \draw[edge] (2)to(4);
            \draw[edge] (3)to(5);
            \draw[edge] (4)to(5);
            \draw[edge] (1)to[in=-30,out=30,loop] (1);
            \node[labelsty,align=center] at (0,-1.5) {196608 gains\\$\sim99.6\%$ rigid};
        \end{scope}
        \begin{scope}[xshift=12cm]
            \node[vertex] (1) at (0:1) {};
            \node[vertex] (2) at (72:1) {};
            \node[vertex] (3) at (144:1) {};
            \node[vertex] (4) at (216:1) {};
            \node[vertex] (5) at (288:1) {};
            \draw[edge] (1)to[bend right=15] (2);
            \draw[edge] (1)to[bend right=-15] (2);
            \draw[edge] (1)to(3);
            \draw[edge] (1)to(4);
            \draw[edge] (1)to(5);
            \draw[edge] (2)to(3);
            \draw[edge] (2)to(4);
            \draw[edge] (2)to(5);
            \draw[edge] (1)to[in=-30,out=30,loop] (1);
            \node[labelsty,align=center] at (0,-1.5) {147456 gains\\$100\%$ rigid};
        \end{scope}
        \begin{scope}[yshift=-3cm]
            \begin{scope}
                \node[vertex] (1) at (0:1) {};
                \node[vertex] (2) at (60:1) {};
                \node[vertex] (3) at (120:1) {};
                \node[vertex] (4) at (180:1) {};
                \node[vertex] (5) at (240:1) {};
                \node[vertex] (6) at (300:1) {};
                \draw[edge] (1)edge(2) (1)edge(3) (1)edge(4) (1)edge(5) (1)edge(6);
                \draw[edge] (4)edge(2) (4)edge(3) (4)edge(5) (4)edge(6);
                \draw[edge] (2)edge(3) (5)edge(6);
                \node[labelsty,align=center] at (0,-1.5) {4194304 gains\\$\sim96.8\%$ rigid};
            \end{scope}
            \begin{scope}[xshift=4cm]
                \node[vertex] (1) at (0:1) {};
                \node[vertex] (5) at (60:1) {};
                \node[vertex] (2) at (120:1) {};
                \node[vertex] (4) at (180:1) {};
                \node[vertex] (3) at (240:1) {};
                \node[vertex] (6) at (300:1) {};
                \draw[edge] (1)edge(2) (1)edge(3) (1)edge(4) (1)edge(5) (1)edge(6) (2)edge(3) (2)edge(4) (2)edge(5) (3)edge(4) (3)edge(6) (5)edge(6);
                \node[labelsty,align=center] at (0,-1.5) {4194304 gains\\$\sim97.8\%$ rigid};
            \end{scope}
            \begin{scope}[xshift=8cm]
                \node[vertex] (1) at (0:1) {};
                \node[vertex] (2) at (60:1) {};
                \node[vertex] (3) at (120:1) {};
                \node[vertex] (4) at (180:1) {};
                \node[vertex] (5) at (240:1) {};
                \node[vertex] (6) at (300:1) {};
                \draw[edge] (1)edge(2) (1)edge(3) (1)edge(4) (1)edge(5) (1)edge(6) (2)edge(3) (2)edge(4) (2)edge(5) (3)edge(6) (4)edge(6) (5)edge(6);
                \node[labelsty,align=center] at (0,-1.5) {4194304 gains\\$\sim98.7\%$ rigid};
            \end{scope}
             \begin{scope}[xshift=12cm]
                \node[vertex] (3) at (0:1) {};
                \node[vertex] (2) at (60:1) {};
                \node[vertex] (6) at (120:1) {};
                \node[vertex] (4) at (180:1) {};
                \node[vertex] (5) at (240:1) {};
                \node[vertex] (1) at (300:1) {};
                \draw[edge] (1)edge(2) (1)edge(3) (1)edge(4) (1)edge(5) (2)edge(3) (2)edge(4) (2)edge(6) (3)edge(5) (3)edge(6) (4)edge(5) (4)edge(6);
                \node[labelsty,align=center] at (0,-1.5) {4194304 gains\\$\sim98.9\%$ rigid};
            \end{scope}
        \end{scope}
    \end{tikzpicture}
    \caption{The percentage of gain maps that are rigid gain assignments with respect to the group of 90 degree rotations in $\mathbb{R}^2$, for a selection of $(2,1)$-tight multigraphs.}
    \label{fig:dim2rot}
\end{figure}

\Cref{fig:dim3rot} shows a similar analysis for 180 degree rotations in three-dimensional space. Here the underlying multigraph should be $(3,0)$-tight. We were able to compute all probabilities for multigraphs with at most four vertices. While there are only 3440 non-isomorphic $(3,0)$-tight multigraphs on four vertices, in the worst case there are $4^{12}$ (over 16 million) possible gain maps. Across all of the multigraphs we computed, there are no multigraphs with less than $96\%$ rigid gain assignments.

In general such computations are mainly time constrained by the shear number of multigraphs and possible gain maps. Single multigraphs with more vertices could be computed given computational time. However, computing all appropriate multigraphs and checking all possible gain maps for $\Gamma$-symmetric rigidity soon gets infeasible.

\begin{figure}[ht]
    \centering
    \begin{tikzpicture}[every loop/.style={min distance=6mm,looseness=5}]
        \begin{scope}
            \node[vertex] (1) at (-30:1) {};
            \node[vertex] (2) at (90:1) {};
            \node[vertex] (3) at (210:1) {};
            \draw[edge] (1)to[bend right=15] (2);
            \draw[edge] (1)to[bend right=-15] (2);
            \draw[edge] (1)to[bend right=15] (3);
            \draw[edge] (1)to[bend right=-15] (3);
            \draw[edge] (2)to[bend right=15] (3);
            \draw[edge] (2)to[bend right=-15] (3);
            \draw[edge] (1)to[in=0,out=-60,loop] (1);
            \draw[edge] (2)to[in=60,out=120,loop] (2);
            \draw[edge] (3)to[in=180,out=240,loop] (3);
            \node[labelsty,align=center] at (0,-1.5) {46656 gains\\$\sim96.1\%$ rigid};
        \end{scope}
        \begin{scope}[xshift=4cm]
            \node[vertex] (1) at (-30:1) {};
            \node[vertex] (2) at (90:1) {};
            \node[vertex] (3) at (210:1) {};
            \draw[edge] (1)to[bend right=15] (2);
            \draw[edge] (1)to[bend right=-15] (2);
            \draw[edge] (1)to(2);
            \draw[edge] (1)to(3);
            \draw[edge] (2)to[bend right=15] (3);
            \draw[edge] (2)to[bend right=5] (3);
            \draw[edge] (2)to[bend right=-5] (3);
            \draw[edge] (2)to[bend right=-15] (3);
            \draw[edge] (2)to[in=60,out=120,loop] (2);
            \node[labelsty,align=center] at (0,-1.5) {6912 gains\\$100\%$ rigid};
        \end{scope}
        \begin{scope}[xshift=8cm]
            \node[vertex] (1) at (-30:1) {};
            \node[vertex] (2) at (90:1) {};
            \node[vertex] (3) at (210:1) {};
            \draw[edge] (1)to[bend right=15] (2);
            \draw[edge] (1)to[bend right=-15] (2);
            \draw[edge] (1)to(2);
            \draw[edge] (1)to[bend right=15] (3);
            \draw[edge] (1)to[bend right=-15] (3);
            \draw[edge] (2)to(3);
            \draw[edge] (2)to[bend right=15] (3);
            \draw[edge] (2)to[bend right=-15] (3);
            \draw[edge] (2)to[in=60,out=120,loop] (2);
            \node[labelsty,align=center] at (0,-1.5) {20736 gains\\$100\%$ rigid};
        \end{scope}
    \end{tikzpicture}
    \begin{tikzpicture}[every loop/.style={min distance=6mm,looseness=5}]
        \begin{scope}[xshift=0cm]
            \node[vertex] (1) at (-45:1) {};
            \node[vertex] (2) at (45:1) {};
            \node[vertex] (3) at (135:1) {};
            \node[vertex] (4) at (-135:1) {};
            \draw[edge] (1)to[bend right=15] (2);
            \draw[edge] (1)to[bend right=-15] (2);
            \draw[edge] (1)to[bend right=15] (3);
            \draw[edge] (1)to[bend right=-15] (3);
            \draw[edge] (1)to[] (4);
            \draw[edge] (2)to[bend right=15] (3);
            \draw[edge] (2)to[bend right=-15] (3);
            \draw[edge] (2)to[bend right=15] (4);
            \draw[edge] (2)to[bend right=-15] (4);
            \draw[edge] (3)to[] (4);
            \draw[edge] (2)to[in=15,out=75,loop] (2);
            \draw[edge] (3)to[in=105,out=165,loop] (3);
            \node[labelsty,align=center] at (0,-1.5) {2985984 gains\\$\sim98.4\%$ rigid};
        \end{scope}
        \begin{scope}[xshift=4cm]
            \node[vertex] (1) at (-45:1) {};
            \node[vertex] (2) at (45:1) {};
            \node[vertex] (3) at (135:1) {};
            \node[vertex] (4) at (-135:1) {};
            \draw[edge] (1)to[bend right=15] (2);
            \draw[edge] (1)to[bend right=-15] (2);
            \draw[edge] (1)to[bend right=15] (4);
            \draw[edge] (1)to[bend right=-15] (4);
            \draw[edge] (2)to[bend right=15] (3);
            \draw[edge] (2)to[bend right=-15] (3);
            \draw[edge] (3)to[bend right=15] (4);
            \draw[edge] (3)to[bend right=-15] (4);
            \draw[edge] (2)to[in=15,out=75,loop] (2);
            \draw[edge] (3)to[in=60,out=120,loop] (3);
            \draw[edge] (3)to[in=150,out=210,loop] (3);
            \draw[edge] (4)to[in=-105,out=-165,loop] (4);
            \node[labelsty,align=center] at (0,-1.5) {1119744 gains\\$\sim99.8\%$ rigid};
        \end{scope}
        \begin{scope}[xshift=8cm]
            \node[vertex] (1) at (-45:1) {};
            \node[vertex] (2) at (45:1) {};
            \node[vertex] (3) at (135:1) {};
            \node[vertex] (4) at (-135:1) {};
            \draw[edge] (1)to[] (2);
            \draw[edge] (1)to[bend right=15] (2);
            \draw[edge] (1)to[bend right=-15] (2);
            \draw[edge] (1)to[] (3);
            \draw[edge] (2)to[] (3);
            \draw[edge] (2)to[bend right=-15] (4);
            \draw[edge] (2)to[] (4);
            \draw[edge] (2)to[bend right=15] (4);
            \draw[edge] (3)to[in=105,out=165,loop] (3);
            \draw[edge] (3)to[in=45,out=105,loop] (3);
            \draw[edge] (3)to[in=165,out=225,loop] (3);
            \node[labelsty,align=center] at (0,-1.5) {55296 gains\\$100\%$ rigid};
        \end{scope}
        \begin{scope}[xshift=12cm]
            \node[vertex] (1) at (-45:1) {};
            \node[vertex] (2) at (45:1) {};
            \node[vertex] (3) at (135:1) {};
            \node[vertex] (4) at (-135:1) {};
            \draw[edge] (1)to[bend right=15] (2);
            \draw[edge] (1)to[bend right=-15] (2);
            \draw[edge] (1)to[] (3);
            \draw[edge] (1)to[] (4);
            \draw[edge] (2)to(3);
            \draw[edge] (2)to[bend right=15] (3);
            \draw[edge] (2)to[bend right=-15] (3);
            \draw[edge] (2)to(4);
            \draw[edge] (3)to(4);
            \draw[edge] (2)to[in=30,out=-30,loop] (2);
            \draw[edge] (2)to[in=60,out=120,loop] (2);
            \draw[edge] (3)to[in=105,out=165,loop] (3);
            \node[labelsty,align=center] at (0,-1.5) {1327104 gains\\$100\%$ rigid};
        \end{scope}
    \end{tikzpicture}
    \caption{The percentage of gain maps that are rigid gain assignments with respect to the group of 180 degree rotations around the three axes in 
    $\mathbb{R}^3$, for a selection of $(3,0)$-tight multigraphs.}
    \label{fig:dim3rot}
\end{figure}
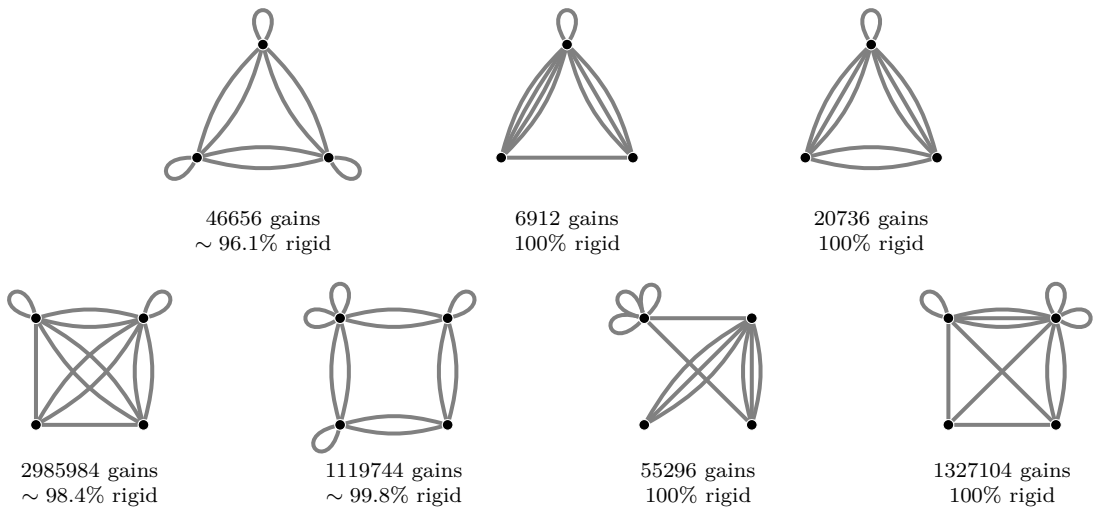

In spite of these computational results, it is actually possible to construct multigraphs, for any finite symmetry group $\Gamma \leq O(d)$, for which the probability of choosing a $\Gamma$-symmetrically rigid gain assignment is positive and close to zero. We prove this below. To do that we first consider the effect of extension operations on this probability.

\subsection{Probabilities under extension operations}

We begin with the $d$-dimensional analogue of the 0-extension introduced in \Cref{sec:plane-1iso} which produces no parallel edges.

\begin{proposition}\label{p:0extprob}
    Let $\Gamma \leq O(d)$ be a finite group and $G=(V,E)$ be a multigraph with at least $d$ vertices.
    Let $G'=(V',E')$ be a multigraph formed from $G$ by adding a vertex and connecting it to $d$ distinct vertices.
    Then $\mathbb{P}(G,\Gamma) = \mathbb{P}(G',\Gamma)$.
\end{proposition}

\begin{proof}
    Fix a gain map $\phi$ of $G$ and let $p$ be a regular orbit placement of $(G,\phi)$.
    Since the group $\Gamma$ is finite,
    we choose $p$ so that for any $d$ vertices $w_1,\ldots,w_d \in V$ and any $d$ gains $\mu_1,\ldots,\mu_d \in \Gamma$,
    the set $\{ \mu_1 p(w_1),\ldots, \mu_d p(w_d)\}$ is linearly independent.
    Let $v_0$ be the vertex added to $G$ to $G'$ that has neighbours $v_1,\ldots,v_d \in V$.
    Now choose any gain map $\phi'$ of $G'$ so that $\phi'(e,v,w) = \phi(e,v,w)$ for all edges $(e,v,w) \in \vec{E}$,
    and define $p'$ to be the orbit placement of $G'$ with $p'(v_0) = 0$ and $p'(v) = p(v)$ for all $v \in V$.
    For each $i \in \{1,\ldots,d\}$,
    fix $\gamma_i$ to be the gain of the edge $e_i$ from $v_0$ to $v_1$ in $G'$,
    i.e., $\gamma_i := \phi'(e_i,v_0,v_i)$.
    We now note that the orbit rigidity matrix of $(G',\phi',p')$ is of the form
    \begin{align*}
        \begin{bmatrix}
            R(G,\phi,p) & \mathbf{0}_{|E| \times d} \\
            A & B 
        \end{bmatrix},
    \end{align*}
    where $\mathbf{0}_{|E| \times d}$ is the $|E| \times d$ zero matrix,
    $A$ is a $d \times d|V|$ matrix and $B$ is the $d \times d$ square matrix
    \begin{align*}
        \begin{bmatrix}
            - \gamma_1 p(v_1) \\
            \vdots \\
            - \gamma_d p(v_d)
        \end{bmatrix}.
    \end{align*}
    By our choice of $p$,
    the matrix $B$ has independent rows,
    hence $R(G',\phi',p') = R(G,\phi,p) + d$.
    We note that due to the layout of the orbit rigidity matrix for $(G',\phi',p')$,
    the orbit framework $(G',\phi',p')$ is regular.
    It now follows that $(G',\phi')$ is $\Gamma$-symmetrically rigid if and only if $(G,\phi)$ is $\Gamma$-symmetrically rigid.
    
    For every gain map $\phi$ of $G$,
    there exist exactly $|\Gamma|^d$ gain maps $\phi'$ of $G$ that agree with $\phi$ on the edges of $G$,
    and, as we have just proven,
    each gain graph $(G',\phi')$ is $\Gamma$-symmetrically rigid if and only if $(G,\phi)$ is.
    Hence $\mathbb{P}(G,\Gamma) = \mathbb{P}(G',\Gamma)$.
\end{proof}

It is important in \Cref{p:0extprob} that the $d$ vertices are distinct.
For example,
let $\Gamma$ be the symmetry group of the 4-dimensional hypercube; i.e., the group of all matrices with exactly one non-zero entry per row and column, which can be either $1$ or $-1$.
Let $G$ be a multigraph with $0 < \mathbb{P}(G,\Gamma)<1$ and let $G'$ be a multigraph formed from $G$ by adding a new vertex $v_0$ and connecting to one vertex $v_1$ of $G$ by four parallel edges $e_1,e_2,e_3,e_4$.
Fix the elements
\begin{align*}
    \gamma_1:=
    \begin{bmatrix}
        1 & 0 & 0 & 0 \\
        0 & 1 & 0 & 0 \\
        0 & 0 & 1 & 0 \\
        0 & 0 & 0 & 1
    \end{bmatrix},
    \,
    \gamma_2:=
    \begin{bmatrix}
        -1 & 0 & 0 & 0 \\
        0 & 1 & 0 & 0 \\
        0 & 0 & 1 & 0 \\
        0 & 0 & 0 & 1
    \end{bmatrix},
    \,
    \gamma_3:=
    \begin{bmatrix}
        1 & 0 & 0 & 0 \\
        0 & -1 & 0 & 0 \\
        0 & 0 & 1 & 0 \\
        0 & 0 & 0 & 1
    \end{bmatrix},
    \,
    \gamma_4:=
    \begin{bmatrix}
        -1 & 0 & 0 & 0 \\
        0 & -1 & 0 & 0 \\
        0 & 0 & 1 & 0 \\
        0 & 0 & 0 & 1
    \end{bmatrix}
\end{align*}
Now fix a gain map $\phi$ of $G$ and a regular orbit placement $p$ of $(G,\phi)$ with $p(v_1) = (w,x,y,z) \neq (0,0,0,0)$.
From this we define the gain map $\phi'$ of $G'$ with $\phi'(e,v,w) = \phi(e,v,w)$ for all $(e,v,w) \in \vec{E}$ and $\phi'(e_i,v_0,v_1) = \gamma_i$ for each $i \in \{1,2,3,4\}$,
and we define the orbit placement $p'$ of $(G',\phi')$ with $p'(v) = p(v)$ for all $v \in V$ and $p'(v_0) = (0,0,0,0)$.
We now note that 
\begin{align*}
    R(G',\phi',p') :=
    \begin{bmatrix}
        R(G,\phi,p) & \mathbf{0}_{|E| \times 4} \\
        A & -B 
    \end{bmatrix},
\end{align*}
where $\mathbf{0}_{|E| \times 4}$ is the $|E| \times 4$ zero matrix,
$A$ is a $4 \times 4|V|$ matrix and $B$ is the $4 \times 4$ square matrix
\begin{align*}
    \begin{bmatrix}
        w & x & y & z\\
        -w & x & y & z\\
        w & -x & y & z\\
        -w & -x & y & z
    \end{bmatrix}.
\end{align*}
As $\rank B = 3$,
it follows that $(G',\phi',p')$ is $\Gamma$-symmetrically flexible regardless of whether $(G,\phi,p)$ is $\Gamma$-symmetrically rigid or not.
Hence $\mathbb{P}(G,\Gamma) > \mathbb{P}(G',\Gamma)$.

A similar construction for $d = 2$ is impossible.
In fact, we can even extend \Cref{p:0extprob} so that it covers some other graph extension moves also.

\begin{proposition}\label{p:2Dextprob}
    Let $\Gamma \leq O(2)$ be a finite group and $G=(V,E)$ be a multigraph.
    Let $G'=(V',E')$ be a multigraph formed from $G$ by a 0-extension or a loop-1-extension.
    Then $\mathbb{P}(G,\Gamma) = \mathbb{P}(G',\Gamma)$.
\end{proposition}

\begin{proof}
    If $G'$ is formed from $G$ by a 0-extension of $G$ that adds two edges that are not parallel,
    then $\mathbb{P}(G,\Gamma) = \mathbb{P}(G',\Gamma)$ by \Cref{p:0extprob}.
    For the other cases,
    we always be adding a vertex $v_0$ which is adjacent to exactly one vertex $v_1$ of $G$.
    Fix a gain map $\phi$ of $G$ and an orbit placement $p$ of the gain graph $(G,\phi)$;
    we assume the latter has the added property that for every distinct pair $\gamma,\mu \in \Gamma$ we have $\gamma p(v_1) \neq \mu p(v_1)$.
    We now split into the two remaining cases.
    
    ($G'$ is a 0-extension of $G$):
    Let $e_1,e_2$ be the two edges joining $v_0$ and $v_1$ in $G'$.
    Choose any gain map $\phi'$ of $G'$ so that $\phi'(e,v,w) = \phi(e,v,w)$ for all $(e,v,w) \in \vec{E}$.
    Since $\phi'$ is a gain map,
    the group elements $\gamma_1 := \phi'(e_1,v_0,v_1)$ and $\gamma_2 := \phi'(e_2,v_0,v_1)$ must be distinct,
    and hence $\gamma_1 p(v_1) \neq \gamma_2 p(v_2)$.
    Choose an orbit placement $p'$ of $(G',\phi')$ where $p'(v) = p(v)$ for all $v \in V$ and $p'(v_0)$ does not lie on the line through the points $\gamma_1 p(v_1),\gamma_2 p(v_2)$.
    We now note that the orbit rigidity matrix of $(G',\phi',p')$ is of the form
    \begin{align}\label{eq:matrix}
        \begin{bmatrix}
            R(G,\phi,p) & \mathbf{0}_{|E| \times 2} \\
            A & B 
        \end{bmatrix},
    \end{align}
    where $\mathbf{0}_{|E| \times 2}$ is the $|E| \times 2$ zero matrix,
    $A$ is a $2 \times 2|V|$ matrix and $B$ is the $2 \times 2$ square matrix
    \begin{align*}
        \begin{bmatrix}
            p(v_0) - \gamma_1 p(v_1) \\
            p(v_0) - \gamma_2 p(v_2)
        \end{bmatrix}.
    \end{align*}
    By our choice of $p$,
    the matrix $B$ has independent rows,
    hence $R(G',\phi',p') = R(G,\phi,p) + 2$.
    It now follows that $(G',\phi')$ is $\Gamma$-symmetrically rigid if and only if $(G,\phi)$ is $\Gamma$-symmetrically rigid.
    We note that due to the layout of the orbit rigidity matrix for $(G',\phi',p')$,
    the orbit framework $(G',\phi',p')$ is regular.
    
    For every gain map $\phi$ of $G$,
    there exist exactly $|\Gamma|(|\Gamma|-1)$ gain maps $\phi'$ of $G$ that agree with $\phi$ on the edges of $G$,
    and, as we have just proven,
    each gain graph $(G',\phi')$ is $\Gamma$-symmetrically rigid if and only if $(G,\phi)$ is.
    Hence $\mathbb{P}(G,\Gamma) = \mathbb{P}(G',\Gamma)$.
    
    ($G'$ is a loop-1-extension of $G$):
    Let $\ell$ be the loop at $v_0$ in $G'$ and $e_1$ be the edge joining $v_0$ and $v_1$ in $G'$.
    Choose any gain map $\phi'$ of $G'$ so that $\phi'(e,v,w) = \phi(e,v,w)$ for all $(e,v,w) \in \vec{E}$.
    Fix the group elements $\gamma_1 := \phi'(e_1,v_0,v_1)$ and $\gamma_\ell := \phi'(\ell,v_0,v_0)$.
    Choose an orbit placement $p'$ of $(G',\phi')$ where $p'(v) = p(v)$ for all $v \in V$ and $p'(v_0)$ is chosen so that $p'(v_0) - \gamma_1 p'(v_1)$ and $2p'(v_0) - \gamma_\ell p'(v_0) - \gamma_\ell^{-1} p'(v_0)$ are linearly independent;
    this can be seen to be possible as $(2I - \gamma_\ell - \gamma_\ell^{-1})$ is never the zero matrix.
    We now note that the orbit rigidity matrix of $(G',\phi',p')$ is of the form given in \Cref{eq:matrix},
    except now $B$ is the $2 \times 2$ square matrix
    \begin{align*}
        \begin{bmatrix}
            p(v_0) - \gamma_1 p(v_1) \\
            2p'(v_0) - \gamma_\ell p'(v_0) - \gamma_\ell^{-1} p'(v_0)
        \end{bmatrix}.
    \end{align*}
    By our choice of $p$,
    the matrix $B$ has independent rows,
    hence $R(G',\phi',p') = R(G,\phi,p) + 2$.
    It now follows that $(G',\phi')$ is $\Gamma$-symmetrically rigid if and only if $(G,\phi)$ is $\Gamma$-symmetrically rigid.
    We note that due to the layout of the orbit rigidity matrix for $(G',\phi',p')$,
    the orbit framework $(G',\phi',p')$ is regular.
    
    For every gain map $\phi$ of $G$,
    there exist exactly $|\Gamma|(|\Gamma|-1)$ gain maps $\phi'$ of $G$ that agree with $\phi$ on the edges of $G$,
    and, as we have just proven,
    each gain graph $(G',\phi')$ is $\Gamma$-symmetrically rigid if and only if $(G,\phi)$ is.
    Hence $\mathbb{P}(G,\Gamma) = \mathbb{P}(G',\Gamma)$.
\end{proof}

There is no version of \Cref{p:2Dextprob} for 1-extensions however;
in fact, we cannot even tell if the probability will increase or decrease.
For example,
let $\Gamma$ be the group of all $90^\circ$ rotations in the plane.
\Cref{fig:riggain1ext} (left) depicts a 1-extension where the probability of a gain map describing a $\Gamma$-symmetrically rigid gain graph decreases,
while \Cref{fig:riggain1ext} (right) depicts a 1-extension where the probability increases.

\begin{figure}[ht]
    \centering
    \begin{tikzpicture}[earrow/.style={line width=2pt,-latex}]
        \begin{scope}
            \begin{scope}
                \node[vertex] (1) at (-30:0.8) {};
                \node[vertex] (2) at (90:0.8) {};
                \node[vertex] (3) at (210:0.8) {};
                \node[vertex,rotate around=-60:(1)] (4) at (2) {};
                \draw[edge] (1)to(2);
                \draw[edge] (1)to(3);
                \draw[edge] (2)to[bend right=15] (3);
                \draw[edge] (2)to[bend right=-15] (3);
                \draw[edge] (2)to[bend right=-20] (4);
                \draw[edge] (2)to[bend right=0] (4);
                \draw[edge] (2)to[bend right=20] (4);
                \node[labelsty,align=center] at (0,-1.5) {$100\%$ rigid};
            \end{scope}
            \draw[earrow] (1.5,0)--(2.5,0);
            \begin{scope}[xshift=4cm]
                \node[vertex] (1) at (-30:0.8) {};
                \node[vertex] (2) at (90:0.8) {};
                \node[vertex] (3) at (210:0.8) {};
                \node[vertex,rotate around=-60:(1)] (4) at (2) {};
                \node[vertex,rotate around=60:(3)] (5) at (2) {};
                \draw[edge] (1)to(2);
                \draw[edge] (1)to(3);
                \draw[edge] (2)to(3);
                \draw[edge] (2)to[bend right=-20] (4);
                \draw[edge] (2)to[bend right=0] (4);
                \draw[edge] (2)to[bend right=20] (4);
                \draw[edge] (1)to(5);
                \draw[edge] (2)to(5);
                \draw[edge] (3)to(5);
                \node[labelsty,align=center] at (0,-1.5) {$\sim98.4\%$ rigid};
            \end{scope}
        \end{scope}
        
        \begin{scope}[xshift=9cm]
            \begin{scope}
                \node[vertex] (1) at (0:1) {};
                \node[vertex] (2) at (72:1) {};
                \node[vertex] (3) at (144:1) {};
                \node[vertex] (4) at (216:1) {};
                \node[vertex] (5) at (288:1) {};
                \draw[edge] (1)to(2);
                \draw[edge] (1)to(3);
                \draw[edge] (1)to(4);
                \draw[edge] (1)to(5);
                \draw[edge] (2)to(3);
                \draw[edge] (2)to(4);
                \draw[edge] (2)to(5);
                \draw[edge] (3)to(4);
                \draw[edge] (3)to(5);
                \node[labelsty,align=center] at (0,-1.5) {$\sim96.1\%$ rigid};
            \end{scope}
            \draw[earrow] (1.5,0)--(2.5,0);
            \begin{scope}[xshift=3.6cm]
                \node[vertex] (1) at (0:1) {};
                \node[vertex] (2) at (72:1) {};
                \node[vertex] (3) at (144:1) {};
                \node[vertex] (4) at (216:1) {};
                \node[vertex] (5) at (288:1) {};
                \node[vertex] (6) at (324:1.2) {};
                \draw[edge] (1)to(2);
                \draw[edge] (1)to(3);
                \draw[edge] (1)to(4);
                \draw[edge] (2)to(3);
                \draw[edge] (2)to(4);
                \draw[edge] (2)to(5);
                \draw[edge] (3)to(4);
                \draw[edge] (3)to(5);
                \draw[edge] (1)to(6);
                \draw[edge] (5)to(6);
                \draw[edge] (3)to(6);
                \node[labelsty,align=center] at (0,-1.5) {$\sim97.8\%$ rigid};
            \end{scope}
        \end{scope}
    \end{tikzpicture}
    \caption{1-extensions with different effects on the relative number of rigid gains.}
    \label{fig:riggain1ext}
\end{figure}
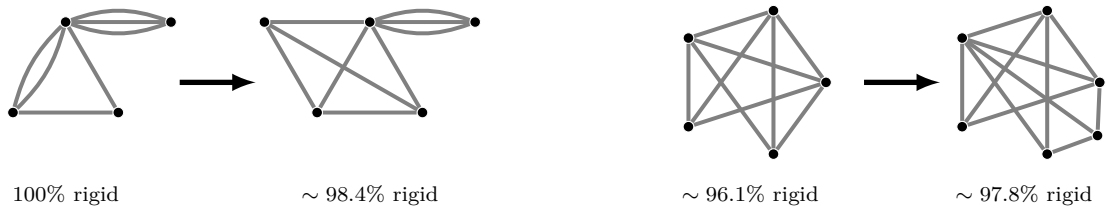

\subsection{Approximating probabilities}

All of the examples illustrated above have a very large percentage of rigid gain assignments. We shall prove that the other extreme is possible.

\begin{theorem}\label{thm:qepsilon}
    Let $\Gamma \leq O(d)$ a finite group and $q \in (0,1)$.
    For every $\varepsilon >0$,
    there exists $N \in \mathbb{N}$ so that for each $n \geq N$ there exists a multigraph on $n$ vertices with $|\mathbb{P}(G,\Gamma) -q| < \varepsilon$.
\end{theorem}

We first note that we can multiply the probabilities for any pair of multigraphs with the following two results.

\begin{lemma}\label{lem:kedges}
    Let $\Gamma$ be a subgroup of $O(d)$ and  $k$ be the dimension of the group of trivial isometries for all $\Gamma$-symmetric frameworks.
    Let $(G_1,\phi_1)$ and $(G_2,\phi_2)$ be $\Gamma$-gain graphs with at least $\max \{d+1,k\}$ vertices each.
    Let $(G,\phi)$ be a $\Gamma$-gain graph formed by joining $(G_1,\phi_1)$ and $(G_2,\phi_2)$ by $k$ independent gained edges.
    Then $(G,\phi)$ is $\Gamma$-symmetrically rigid in $\mathbb{R}^d$ if and only if both $(G_1,\phi_1)$ and $(G_2,\phi_2)$ are $\Gamma$-symmetrically rigid in $\mathbb{R}^d$.
\end{lemma}

\begin{proof}
    We first note that 
    we may assume that the gain of each edge connecting $G_1$ to $G_2$ has trivial gain\footnote{It is well known that one may multiply the gains on all edges incident to a vertex by a common group element without changing the $\Gamma$-symmetric rigidity \cite{jordkasztani16}.}.
    Let $v_1w_1,\ldots,v_kw_k$ be the $k$ independent edges between $G_1$ and $G_2$,
    where each $v_i$ lies in $G_1$ and each $w_i$ lies in $G_2$.
    Choose a regular orbit placement $p_1$ of $(G_1,\phi_1)$.
    Let $u^1,\ldots,u^k$ be a basis of the trivial infinitesimal of $(G_1,\phi_1,p_1)$.
    We may assume by perturbing vertices that for every element $u^i$ of the basis we have $u^i(v) \neq 0$ for each $v \in V(G_1)$;
    this is as trivial infinitesimal flexes are defined by multiplying each $p(v)$ by a skew-symmetric matrix that is non-zero on an open dense set of points.
    Define for each $i \in \{1,\ldots,k\}$ the open dense subset $W_i := \{ x \in \mathbb{R}^d : p(v_i) . u(v_i) \neq x . u(v_i) \}$.
    Now define $p_2$ to be a regular orbit placement of $(G_2,\phi_2)$ in general position where $p_2(v) \neq p_2(w)$ for all $v \in V(G_1)$ and $w \in V(G_2)$,
    and $p_2(w_i) \in W_i$ for each $i \in \{1,\ldots,k\}$.
    Now define $(G,\phi,p)$ to be the regular framework with $p(v) = p_i(v)$ for all $v \in V(G_i)$;
    we may assume $p$ is regular by perturbing the orbit placements $p_1$ and $p_2$ whilst keeping all the previous properties discussed.
    
    Define $K_i := \ker R(G_i,\phi_i,p_i)$ and $K:= \ker R(G,\phi,p)$.
    Any infinitesimal flex of $(G,\phi,p)$ is a subset of $K_1 \oplus K_2$.
    We note that as the $k$ edges $v_1w_1,\ldots,v_kw_k$ form $k$ linear constraints,
    we have $\dim K_1 + \dim K_2 - k \leq \dim K \leq \dim K_1 + \dim K_2$
    It now suffices to prove that there are $k$ linearly independent elements in $K_1 \oplus K_2$ that are not in $K$.
    By our choice of framework,
    each of the elements $u^i \oplus 0 \in K_1 \oplus K_2$ does not lie in $K$,
    hence the result holds.
\end{proof}

\begin{lemma}\label{lem:smallprob}
    Let $\Gamma \leq O(d)$ be a finite group.
    For any pair of multigraphs $G_1,G_2$, with at least $d$ vertices,
    there exists a multigraph $G$ where $\mathbb{P}(G,\Gamma) = \mathbb{P}(G_1,\Gamma) \mathbb{P}(G_2, \Gamma)$.
\end{lemma}

\begin{proof}
   Define $k$ to be the dimension of the group of trivial isometries for all $\Gamma$-symmetric frameworks.
   By applying \Cref{p:0extprob},
   we may assume that both $G_1$ and $G_2$ have at least $\max\{d+1,k\}$ vertices each.
   Let $G$ be the multigraph formed from joining $G_1$ and $G_2$ by $k$ independent edges (see \Cref{fig:edgeconnect}).
   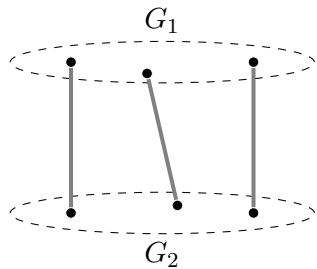
\begin{figure}[ht]
       \centering
        \begin{tikzpicture}[yscale=0.5]
            \node[vertex] (h11) at (-1.2,2) {};
            \node[vertex] (h12) at (1.2,2) {};
            \node[vertex] (h13) at (-0.2,1.7) {};
            \node[vertex] (h21) at (-1.2,-2) {};
            \node[vertex] (h22) at (1.2,-2) {};
            \node[vertex] (h23) at (0.2,-1.8) {};
            \draw[dashed] ($(h11)!0.5!(h12)$) circle [x radius=2cm,y radius=0.55cm];
            \draw[dashed] ($(h21)!0.5!(h22)$) circle [x radius=2cm,y radius=0.55cm];
            \draw[edge] (h11)edge(h21) (h12)edge(h22) (h13)edge(h23);
            \node[] at ($(h11)!0.5!(h12)+(0,1.1)$) {$G_1$};
            \node[] at ($(h21)!0.5!(h22)+(0,-1.1)$) {$G_2$};
        \end{tikzpicture}
        \caption{Joining two multigraphs $G_1$ and $G_2$ by $k$ independent edges when $k=3$.}
        \label{fig:edgeconnect}
   \end{figure}
  
   If we choose any $\Gamma$-gain map $\phi$ for $G$ and define $\phi_i$ to be the restriction of $\phi$ to the edges of $G_i$,
   then $(G,\phi)$ is $\Gamma$-symmetrically rigid if and only if both $(G_1,\phi_1)$ and $(G_2,\phi_2)$ are $\Gamma$-symmetrically rigid by \Cref{lem:kedges}.
   Hence, $\mathbb{P}(G,\Gamma) = \mathbb{P}(G_1,\Gamma) \mathbb{P}(G_2, \Gamma)$ as required.
\end{proof}

We now wish to construct a multigraph with an arbitrarily high probability.
To do so we first need the following result.

\begin{lemma}\label{lem:gammah}
    For each finite subgroup $\Gamma \leq O(d)$, there exists a simple graph $H$ with at least $d+1$ vertices where $0< \mathbb{P}(H,\Gamma) <1$.
\end{lemma}

\begin{proof}
    Define $(H',\phi')$ to be the complete loopless $\Gamma$-gain graph with $d+1$ vertices,
    i.e.~$H'$ has vertices $v_0,\dots,v_d$,
    and for every pair of distinct vertices $v_i,v_j$ and every $\gamma \in \Gamma$,
    there exists an edge $(e,v_i,v_j)$ with $\phi'(e,v_i,v_j)= \gamma$ (unless $\gamma$ is trivial and $v_i=v_j$).
    Choose $p'$ to be a $\Gamma$-symmetrically rigid orbit placement of $(H',\phi')$ where the points $p(v_0),\ldots,p(v_d)$ are affinely independent.
    Using methods similar to those in the proof of \Cref{p:0extprob},
    we see that $(H',\phi',p')$ is $\Gamma$-symmetrically rigid.
    
    We now define the $\Gamma$-symmetric framework $(H,\phi,p)$ in the following way.
    The graph $H =(V,E)$ is a simple graph with
    \begin{align*}
        V := \{v_0,\ldots,v_d\} ~ \cup ~ \{ v_{i,\gamma} : 0\leq i \leq d, ~ \gamma \in \Gamma \}
    \end{align*}
    and, with $I$ denoting the $d \times d$ identity matrix,
    \begin{align*}
        E := &\{ v_{i,I} v_{j,\gamma} : 0\leq i, j \leq d,~ i \neq j, ~ \gamma \in \Gamma\} ~ \cup\\
             &\{ v_{i} v_{j} : 0\leq i <j \leq d \} ~ \cup \\
             &\{ v_{i} v_{j,\gamma} : 0\leq i ,j \leq d,~ i \neq j,  ~ \gamma \in \Gamma\}.
    \end{align*}
    To define the gain map $\phi$,
    for an edge $(e,v_{i,I},v_{j,\gamma})$ we set $\phi (e,v_{i,I},v_{j,\gamma}) = \gamma$. The remaining edges are assigned trivial gains.
    See \Cref{fig:gammah} for an example of $(H,\phi)$ for the case when $d=2$ and $\Gamma = \{I , -I\}$.
    Finally,
    we set $p(v_{i}) = p'(v_i)$ for every $v_{i} \in V$ and $p(v_{i,\gamma}) = p'(v_i)$ for every $v_{i,\gamma} \in V$.
    
    Let $u$ be a $\Gamma$-symmetric flex of $(H,\phi,p)$.
    For any $0\leq i ,j \leq d$ and $\gamma \in \Gamma$,
    the edge $v_{i,\gamma} v_j$ has trivial gain,
    hence $(p(v_{i,\gamma})-p(v_j))\cdot (u(v_{i,\gamma}) - u(v_j)) = 0$.
    By using the substitution $p(v_{i,\gamma})=p(v_i)$ and rearranging for each $j$, we see that $u(v_{i,\gamma})$ is the unique element that satisfies the $d$ affine conditions
    \begin{align*}
         (p(v_i)-p(v_j))\cdot u(v_{i,\gamma}) = (p(v_i)-p(v_j))\cdot u(v_j) \qquad \text{for all } 0 \leq j \leq d,~j\neq i.
    \end{align*}
    Hence, $u(v_{i,\gamma}) = u(v_{i})$ for all $\gamma \in \Gamma$.
    Choose any $0 \leq i,j \leq d$, $i \neq j$, and $\gamma \in \Gamma$.
    As $(v_{i,I} v_{j,\gamma}, v_{i,I}, v_{j,\gamma})$ is an edge with gain $\gamma$, it follows that
    \begin{align*}
         (p(v_i)- \gamma p(v_j))\cdot (u(v_i) - u(v_j)) = (p(v_{i,I})- \gamma p(v_{j,\gamma}))\cdot (u(v_{i,I}) - u(v_{j,\gamma})) = 0.
    \end{align*}
    Define the map $u: V(H') \rightarrow \mathbb{R}^d$ by setting $u'(v_i) := u(v_i)$.
    By our choice of $(H,\phi,p)$,
    $u'$ is a $\Gamma$-symmetric flex of $(H',\phi',p')$,
    hence there exists a skew-symmetric matrix $A$ and a point $x \in \mathbb{R}^d$ so that $u'(v_i) = Ap'(v_i) +x$.
    It follows that $u(v_i) = Ap(v_i) +x$ and $u(v_{i,\gamma}) = A p(v_{i,\gamma}) +x$ for all vertices $v_i,v_{i,\gamma}$ of $H$,
    hence $(H,\phi,p)$ is $\Gamma$-symmetrically rigid.
    Since $H$ is simple,
    $\mathbb{P}(H,\Gamma) <1$,
    thus $0< \mathbb{P}(H,\Gamma) <1$.
\end{proof}

\begin{figure}[ht]
	\begin{center}
		\begin{tikzpicture}[scale=1.7]
			\node[vertex,label={[labelsty]above:$v_0$}] (0) at (0,1) {};
			\node[vertex,label={[labelsty,label distance=-2pt]210:$v_1$}] (1) at (-0.866,-0.5) {};
			\node[vertex,label={[labelsty,label distance=-2pt]-30:$v_2$}] (2) at (0.866,-0.5) {};
			
 			\node[vertex,label={[labelsty]above:$v_{0,I}$}] (01) at (0,1*2) {};
 			\node[vertex,label={[labelsty,label distance=-2pt]210:$v_{1,I}$}] (11) at (-0.866*2,-0.5*2) {};
 			\node[vertex,label={[labelsty,label distance=-2pt]-30:$v_{2,I}$}] (21) at (0.866*2,-0.5*2) {};
			
 			\node[vertex,label={[labelsty]above:$v_{0,-I}$}] (0m1) at (0,1*3) {};
 			\node[vertex,label={[labelsty,label distance=-2pt]210:$v_{1,-I}$}] (1m1) at (-0.866*3,-0.5*3) {};
 			\node[vertex,label={[labelsty,label distance=-2pt]-30:$v_{2,-I}$}] (2m1) at (0.866*3,-0.5*3) {};
 			
			
			

			\draw[edge] (0)edge(1);
			\draw[edge] (1)edge(2);
			\draw[edge] (0)edge(2);
			
			\draw[edge] (1)edge(01);
			\draw[edge] (2)edge(01);
			
			\draw[edge] (1)edge(0m1);
			\draw[edge] (2)edge(0m1);
			
			\draw[edge] (0)edge(11);
			\draw[edge] (2)edge(11);
			
			\draw[edge] (0)edge(1m1);
			\draw[edge] (2)edge(1m1);
			
			\draw[edge] (0)edge(21);
			\draw[edge] (1)edge(21);
			
			\draw[edge] (0)edge(2m1);
			\draw[edge] (1)edge(2m1);
			
			\draw[edge] (01)edge(11);
			\draw[edge] (11)edge(21);
			\draw[edge] (21)edge(01);
			
			\draw[redge] (01)edge(1m1);
			\draw[redge] (01)edge(2m1);
			
			\draw[redge] (11)edge(0m1);
			\draw[redge] (11)edge(2m1);
			
			\draw[redge] (21)edge(0m1);
			\draw[redge] (21)edge(1m1);
		\end{tikzpicture}
	\end{center}
	\caption{Example of the graph constructed in \Cref{lem:gammah} with $d=2$ and $\Gamma = \{I,-I\}$. Edges with gain $-I$ are marked in red.}
	\label{fig:gammah}
\end{figure}
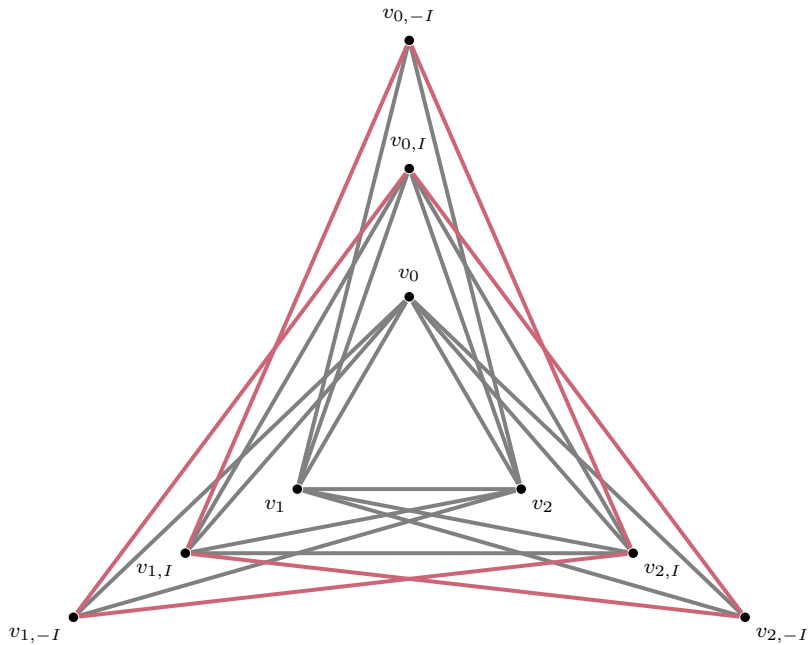

\begin{lemma}\label{lem:bigprob}
    Let $\Gamma \leq O(d)$ a finite group.
    For every $\varepsilon >0$,
    there exists $N \in \mathbb{N}$ so that for each $n \geq N$ there exists a multigraph on $n$ vertices with $1 -\varepsilon < \mathbb{P}(G,\Gamma) < 1$.
\end{lemma}

\begin{proof}
    Fix $\varepsilon$.
    By \Cref{lem:gammah}, there exists a simple graph $\tilde{H}$ where $0< \mathbb{P}(\tilde{H},\Gamma) < 1$ with vertices $\{v_1,\ldots, v_D\}$ (with $D >d$).
    Now define $H$ to be the graph formed from $\tilde{H}$ by adding $d$ vertices $w_1,\ldots,w_d$ and every edge of the form $w_i v_j$ for $1 \leq i \leq d$ and $1 \leq j \leq D$.
    Importantly, $0<\mathbb{P}(H,\Gamma) <1$;
    we have that $\mathbb{P}(H,\Gamma) \geq \mathbb{P}(\tilde{H},\Gamma)$ by \Cref{p:0extprob},
    and we must have $\mathbb{P}(H,\Gamma) < 1$ since $H$ is a simple graph.
    Choose $m \in \mathbb{N}$ so that $(1- \mathbb{P}(H,\Gamma))^m < \varepsilon$.
    For each $1\leq i \leq m$,
    define $H_i$ to be a graph isomorphic to $H$ with vertices $\{v_1^i,\ldots, v_D^i, w_1,\ldots,w_d\}$;
    note that we have $V(H_i) \cap V(H_j) = \{w_1,\ldots,w_d\}$ for all $i \neq j$.
    Now define $G_m$ to be the graph formed by taking the union of the graphs $H_1,\ldots,H_m$ (see \Cref{fig:bigprobgraph}).
    
    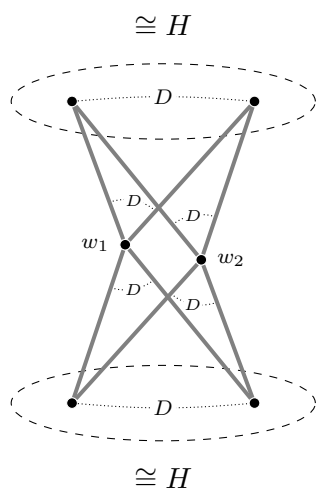
\begin{figure}[ht]
        \centering
        \begin{tikzpicture}[mult/.style={draw,densely dotted,angle radius=0.6cm,font=\tiny,pic text =$D$,angle eccentricity=1,pic text options={circle,fill=white,inner sep=0pt}},multvert/.style={font=\scriptsize,fill=white,inner sep=1pt}]
            \node[vertex,label={[labelsty]left:$w_1$}] (c1) at (-0.5,0.1) {};
            \node[vertex,label={[labelsty]right:$w_2$}] (c2) at (0.5,-0.1) {};
            \node[vertex] (h11) at (-1.2,2) {};
            \node[vertex] (h12) at (1.2,2) {};
            \node[vertex] (h21) at (-1.2,-2) {};
            \node[vertex] (h22) at (1.2,-2) {};
            \draw[dashed] ($(h11)!0.5!(h12)$) circle [x radius=2cm,y radius=0.5cm];
            \draw[dashed] ($(h21)!0.5!(h22)$) circle [x radius=2cm,y radius=0.5cm];
            
            \draw pic [mult] {angle = h12--c1--h11};
            \draw pic [mult] {angle = h12--c2--h11};
            \draw pic [mult] {angle = h21--c1--h22};
            \draw pic [mult] {angle = h21--c2--h22};
            \foreach \p in {h11,h12,h21,h22} 
            {
                \draw[edge] (c1)edge(\p) (c2)edge(\p);
            }
            \draw[densely dotted] (h11)to[bend left=5] node[multvert] {$D$} (h12);
            \draw[densely dotted] (h21)to[bend right=5] node[multvert] {$D$} (h22);
            \node[] at ($(h11)!0.5!(h12)+(0,1)$) {$\cong H$};
            \node[] at ($(h21)!0.5!(h22)+(0,-1)$) {$\cong H$};
        \end{tikzpicture}
        \caption{Construction of $G_m$ for $d=2$.}
        \label{fig:bigprobgraph}
    \end{figure}
    
    Choose any gain map $\phi$ of $G_m$.
    We note that if any gain graph $(H_i,\phi|_{H_i})$ is $\Gamma$-symmetrically rigid, then $(G_m,\phi)$ is $\Gamma$-symmetrically rigid also.
    This holds since every vertex $v_k^j$ is connected to each of the vertices $w_1,\ldots,w_d$.
    Hence, $1- \mathbb{P}(G_m,\Gamma) < (1- \mathbb{P}(H,\Gamma))^m$.
    Since $G_m$ is a simple graph,
    it follows that $1 - \varepsilon < \mathbb{P}(G_m,\Gamma) <1$ as required.
    
    Now set $N = |V(G_m)|$.
    For every $n >N$,
    we can form a graph $G$ with $\mathbb{P}(G,\Gamma) = \mathbb{P}(G_m,\Gamma)$ by applying \Cref{p:0extprob} $n-N$ times.
\end{proof}

We are now ready to prove the main result of this section.

\begin{proof}[Proof of \Cref{thm:qepsilon}]
    Let $\varepsilon>0$. By \Cref{lem:bigprob}, there exists a multigraph $H$ with $\mathbb{P}(H,\Gamma) = 1-\delta$, for some $\delta\in (0,\varepsilon)$. By applying \Cref{lem:smallprob} to copies of $H$,
    we can obtain for every $k\in \mathbb{N}$ a multigraph $G$ with $\mathbb{P}(G,\Gamma) = \mathbb{P}(H,\Gamma)^k = (1-\delta)^k$. Define $a_k=(1-\delta)^k$. Since for every $k$ we have $a_k-a_{k+1}\in (0,\delta)$ and $a_k\rightarrow 0$ as $k \rightarrow \infty$, there exists $m \in \mathbb{N}$ such that $|a_m-q|<\varepsilon$.
    By applying \Cref{lem:smallprob} to copies of $H$,
    we can obtain a multigraph $G$ with $\mathbb{P}(G,\Gamma) = \mathbb{P}(H,\Gamma)^m = (1-\delta)^m$.
    The result now follows by applying \Cref{p:0extprob}.
\end{proof}

\subsection*{Acknowledgements}

Parts of this research was completed during the Fields Institute Thematic Program on Geometric constraint
systems, framework rigidity, and distance geometry. SD and GG were partially supported by the Austrian Science Fund (FWF): P31888.
AN was partially supported by the Heilbronn institute for mathematical research.

\bibliographystyle{plainurl}
\bibliography{lit}

\end{document}